\documentclass[11pt]{amsart}
\usepackage[margin=2.5cm]{geometry}
\usepackage{enumerate}
\usepackage{amsmath}
\usepackage{amssymb,latexsym}
\usepackage{amsthm}
\usepackage{color}
\usepackage[dvipsnames]{xcolor}
\usepackage{cancel}
\usepackage{graphicx}
\usepackage{cite}
\usepackage{changes}

\makeatletter
\renewcommand{\tocsection}[3]{%
  \indentlabel{\@ifnotempty{#2}{\bfseries\ignorespaces#1 #2\quad}}\bfseries#3}
\renewcommand{\tocsubsection}[3]{%
  \indentlabel{\@ifnotempty{#2}{\ignorespaces#1 #2\quad}}#3}

\usepackage{etoolbox}
\makeatletter
\patchcmd{\@setaddresses}{\indent}{\noindent}{}{}
\patchcmd{\@setaddresses}{\indent}{\noindent}{}{}
\patchcmd{\@setaddresses}{\indent}{\noindent}{}{}
\patchcmd{\@setaddresses}{\indent}{\noindent}{}{}
\makeatother

\usepackage[linktocpage]{hyperref}
\hypersetup{
  colorlinks   = true, 
  urlcolor     = blue, 
  linkcolor    = Purple, 
  citecolor   = red 
}
\usepackage{comment}
\usepackage{amscd}
\usepackage{mathtools}
\usepackage{soul}

\DeclareMathOperator{\C}{\mathcal{C}}

\usepackage{cleveref}

\theoremstyle{definition}
\newtheorem{theorem}{Theorem}[section]

\newtheorem{lemma}[theorem]{Lemma}
\newtheorem{corollary}[theorem]{Corollary}

\newtheorem{proposition}[theorem]{Proposition}

\newtheorem{construction}[theorem]{Construction}

\newtheorem{remark}[theorem]{Remark}

\newcommand{\fqn}{\mathbb{F}_{q^n}}

\newcommand{\F}{{\mathbb F}}

\newcommand{\fq}{{\mathbb F}_{q}}

\newcommand{\la}{\langle}
\newcommand{\ra}{\rangle}

\newcommand{\PG}{\mathrm{PG}}
\newcommand{\N}{\mathrm{N}}

\newcommand\qbin[3]{\left[\begin{matrix} #1 \\ #2 \end{matrix} \right]_{#3}}

\newcommand{\fqm}{\mathbb{F}_{q^m}}

\usepackage{comment}

\title{Clubs in projective spaces and three-weight rank-metric codes}
\date{}

\author[J. Mannaert]{Jonathan Mannaert}
\address{Jonathan Mannaert, \textnormal{Department of Mathematics and Data Science, Vrije Universiteit Brussel, Pleinlaan 2, 1050 Brussel, Belgium}}
\email{Jonathan.Mannaert@vub.be}

\author[P. Santonastaso]{Paolo Santonastaso}
\address{Paolo Santonastaso, \textnormal{Dipartimento di Matematica e Fisica, Universit\`a degli Studi della Campania ``Luigi Vanvitelli'', Viale Lincoln, 5, I--\,81100 Caserta, Italy}\newline
\textnormal{Dipartimento di Meccanica, Matematica e Management, Politecnico di Bari, Via Orabona 4, 70125 Bari, Italy}}
\email{paolo.santonastaso@unicampania.it,paolo.santonastaso@poliba.it}

\author[F. Zullo]{Ferdinando Zullo}
\address{Ferdinando Zullo, \textnormal{Dipartimento di Matematica e Fisica, Universit\`a degli Studi della Campania ``Luigi Vanvitelli'', Viale Lincoln, 5, I--\,81100 Caserta, Italy}}
\email{ferdinando.zullo@unicampania.it}

\begin{document}

\begin{abstract}
Linear sets over finite fields are central objects in finite geometry and coding theory, with deep connections to structures such as semifields, blocking sets, KM-arcs, and rank-metric codes. Among them, $i$-clubs, a class of linear sets where all but one point (which has weight $i$) have weight one, have been extensively studied in the projective line but remain poorly understood in higher-dimensional projective spaces.
In this paper, we investigate the geometry and algebraic structure of $i$-clubs in projective spaces. We establish upper bounds on their rank by associating them with rank-metric codes and analyzing their parameters via MacWilliams identities. We also provide explicit constructions of $i$-clubs that attain the maximum rank for $i \geq m/2$, and we demonstrate the existence of non-equivalent constructions when $i \leq m-2$. The special case $i = m-1$ is fully classified. Furthermore, we explore the rich geometry of three-weight rank-metric codes, offering new constructions from clubs and partial classification results.
\end{abstract}

\maketitle


\section{Introduction}

Linear sets have proven to be a powerful tool in various classification results and constructions within finite geometry and coding theory. They have been playing a key role in the study of objects such as semifields, blocking sets, translation ovoids, KM-arcs, and rank-metric codes. For a comprehensive overview of their applications, we refer the reader to~\cite{lavrauw2015field,polverino2010linear}.

Let $\mathbb{V}$ be a $k$-dimensional $\fqm$-vector space, $\Lambda=\PG(\mathbb{V},\fqm)=\PG(k-1,q^m)$ the corresponding projective space and let $U$ be an $\fq$-subspace of $\mathbb{V}$. 
The set of points in the projective space defined by $U$ is denoted by
\[ L_U=\{\langle {u} \rangle_{\fqm} \colon {u}\in U\setminus\{{0}\}\} \]
which is said to be an $\fq$-\emph{linear set} of rank $\dim_{\fq}(U)$.
A central notion associated with linear sets is the \emph{weight} of a point, which intuitively measures how much of the subspace $U$ is ``concentrated'' at that particular point. A linear set is called \emph{scattered} if all of its points have weight one---that is, they intersect the subspace $U$ in the minimal possible way. The concept of scattered linear sets was introduced by Blokhuis and Lavrauw in~\cite{blokhuis2000scattered}, where it was studied in a broader context. These sets have recently attracted considerable interest due to their connection with Maximum Rank Distance (MRD) codes and, more generally, rank-metric codes; see~\cite{polverino2020connections} for a survey.

Closely related to scattered linear sets is the family of linear sets known as \emph{clubs}. An $i$-\emph{club} is an $\fq$-linear set $L_U$ in which all but one of the points have weight one, and the remaining point has weight $i$. 
Clubs on the projective line were originally introduced by Fancsali and Sziklai in~\cite{fancsali2006maximal} (see also~\cite{fancsali2009description}) in the context of maximal partial $2$-spreads in $\PG(8, q)$. Interest in these structures was revived when De Boeck and Van de Voorde in~\cite{de2016linear} characterized translation KM-arcs as those that can be described by $i$-clubs on the projective line in even characteristic. This connection enabled new constructions and classifications of KM-arcs.
Beyond their geometric significance, $i$-clubs also yield linear blocking sets of R\'edei type when interpreted as sets of determined directions of affine point sets. These point sets are also notable for their combinatorial properties and for defining Hamming metric codes with \emph{few weights}; see~\cite{napolitano2023two}.
Moreover, $i$-clubs on the projective line admit a natural algebraic characterization via linearized polynomials. In particular, in~\cite{bartoli2022r} the polynomials associated with clubs in the projective line were studied. Clubs whose associated polynomial is \emph{exceptional} have been shown in~\cite{bartoli2022r} to not exist, indicating that the behavior of such polynomials is deeply dependent on the extension field over which they are considered.
Therefore, their construction heavily depends on the field extension considered, see also \cite{polverino2024fat}. 
Finally, it was recently shown that the free product of rank one uniform $q$-matroids is represented by clubs on the projective line, see \cite{alfarano2024free}.

While clubs on the projective line have been extensively studied, very little is known about clubs in projective space. The aim of this paper is to provide bounds on the rank of an $i$-club, along with their constructions and connections to rank-metric codes.

In Section \ref{sec:prel}, we recall the basics of the theory of linear sets and rank-metric codes. As we will see, rank-metric codes play a fundamental role in the study of such linear sets.
In Section \ref{sec:genclubs}, we describe some general properties of clubs. We begin by presenting known results and constructions, and we prove some basic properties of clubs that will be used later.
Section \ref{sec:bound} is devoted to bounding the rank of an $i$-club. The main idea is to associate a rank-metric code with the dual of an $i$-club and study its parameters. To obtain the final bound, we use the machinery of the MacWilliams identities—this is the first time they have been applied in this context to obtain bounds.
They have also been used to explore the rank-metric code associated with a $2$-club on the line by Sheekey and Van de Voorde in \cite{sheekeyVdV} with a slightly different objective.
In Section \ref{sec:constr}, we provide constructions of $i$-clubs that attain the maximum rank when $i \geq m/2$. We also show that when $i \leq m-2$, there exist non-equivalent constructions.
The case $i = m - 1$ is studied in Section \ref{sec:classm-1club}, where we are able to completely classify the $(m-1)$-clubs with maximum rank.
In Section \ref{sec:threeweights}, we show that the geometry of three-weight rank-metric codes is significantly richer than that of two-weight rank-metric codes. A particularly interesting class of examples arises from the codes associated with the duals of $i$-clubs; we construct new examples and provide some classification results.
We conclude the paper by discussing possible directions for future research.

\section{Preliminaries}\label{sec:prel}

\subsection{Linear sets}

In this paper, we deal with linear sets in projective spaces.
More precisely, let $\mathbb{V}$ be a $k$-dimensional $\fqm$-vector space and let $\Lambda=\PG(\mathbb{V},\fqn)=\PG(k-1,q^m)$. 
Clearly, $\mathbb{V}$ can also be seen as an $\fq$-vector space of dimension $mk$. Therefore, we can consider an $\fq$-subspace $U$ of $\mathbb{V}$. 
The set
\[ L_U=\{\langle {u} \rangle_{\fqm} \colon {u}\in U\setminus\{{0}\}\} \]
is said to be an $\fq$-\emph{linear set} of rank $\dim_{\fq}(U)$.
The rank of $L_U$ will also be denoted by $\mathrm{Rank}(L_U)$.
The \emph{weight} of a projective subspace $S=\PG(W,\fqm) \subseteq \Lambda$ in $L_U$ is defined naturally as $w_{L_U}(S)=\dim_{\fq}(U \cap W)$.

Let us now recall some basic relations between the size of a linear set, the number of points of a certain weight, and its rank.
If $L_U$ has rank $n$ then the weight of any point is bounded by $n$. Denote by $N_i(U)$ the number of points of $\Lambda$ having weight $i\in \{0,\ldots,n\}$ in $L_U$, the following relations hold:
\begin{equation}\label{eq:card}
    |L_U| \leq \frac{q^n-1}{q-1},
\end{equation}
\begin{equation}\label{eq:pesicard}
    |L_U| =N_1(U)+\ldots+N_n(U),
\end{equation}
\begin{equation}\label{eq:pesivett}
    N_1(U)+N_2(U)(q+1)+\ldots+N_n(U)(q^{n-1}+\ldots+q+1)=q^{n-1}+\ldots+q+1.
\end{equation}

An $\fq$-linear set in $\PG(k-1,q^m)$ for which all of its points have weight one is called a \emph{scattered linear set}. If all the points have a weight one except for one that has a weight $i$, it is called an $i$-\emph{club linear set}.

By the above relations, it is easy to see the following.

\begin{proposition}\label{prop:size}
A scattered linear set of rank $n$ has $\frac{q^n-1}{q-1}$ points and an $i$-club of rank $n$ has size $q^{n-1}+\ldots+q^i+1$.
\end{proposition}

Blokhuis and Lavrauw provided the following bound on the rank of a scattered liner set.

\begin{theorem}(see \cite{blokhuis2000scattered})\label{th:rankMaxScattered}
Let $L_U$ be a scattered $\fq$-linear set of rank $n$ in $\PG(k-1,q^m)$, then
\[ n\leq \frac{mk}2. \]
\end{theorem}

A scattered $\fq$-linear set of rank $km/2$ in $\PG(k-1,q^m)$ is said to be \emph{maximum scattered} and $U$ is said to be a maximum scattered $\fq$-subspace.

We recall the following existence result on maximum scattered linear sets in $\PG(k-1,q^m)$, obtained by combining results from a series of papers \cite{blokhuis2000scattered,ball2000linear,bartoli2018maximum,csajbok2017maximum}, that proves the bound in Theorem \ref{th:rankMaxScattered} is tight when $km$ is even.

\begin{theorem}\label{th:existencescattered}
For every positive integers $k$ and $m$ for which $km$ is even, and each value of $q$, there exists a scattered $\fq$-linear set in $\mathrm{PG}(k-1,q^m)$ of rank $\frac{km}2$. 
\end{theorem}

We also recall the characterization of the subspaces whose associated linear set coincides with the entire space.

\begin{proposition}\label{prop:linset=entirespace}(see \cite[Lemma 26]{adriaensen2024minimum})
Let $L_U$ be an $\fq$-linear set of rank $n$ in $\mathrm{PG}(k-1,q^m)$. We have that $L_U=\mathrm{PG}(k-1,q^m)$ if and only if $n\geq m(k-1)+1$.
\end{proposition}

The following mapping describes a duality acting on the $\fq$-subspaces of $\mathbb{V}$ which preserves the $\fqm$-linearity.

Let $\sigma \colon \mathbb{V}\times \mathbb{V} \rightarrow \mathbb{F}_{q^m}$ be a non-degenerate reflexive sesquilinear form over $\mathbb{V}$. Define
$\sigma' \colon \mathbb{V} \times \mathbb{V} \rightarrow \mathbb{F}_q$ by $\sigma':(u,v)\mapsto \mathrm{Tr}_{q^m/q}(\sigma(u,v))$.
If we regard $\mathbb{V}$ as an $\F_q$-vector space, then $\sigma^\prime$ turns out to be a non-degenerate reflexive sesquilinear form on $\mathbb{V}$.
Let $\perp$ and $\perp'$ be the orthogonal complement maps defined by $\sigma$ and $\sigma'$ on the lattices of $\F_{q^m}$-linear and 
$\F_q$-linear subspaces, respectively.
The following properties hold (see \cite[Section~2]{polverino2010linear} for the details).

\begin{proposition}\label{prop:dualityproperties}
With the above notation,
\begin{itemize}
    \item[(i)] $\dim_{\F_{q^m}}(W)+\dim_{\F_{q^m}}(W^\perp)=k$, for every $\F_{q^m}$-subspace $W$ of $\mathbb{V}$.
    \item[(ii)] $\dim_{\F_{q}}(U)+\dim_{\F_{q}}(U^{\perp'})=mk$, for every $\F_{q}$-subspace $U$ of $\mathbb{V}$.
    \item[(iii)] $T_1\subseteq T_2$ implies  $T_1^{\perp'}\supseteq T_2^{\perp'}$, for every $\F_q$-subspaces $T_1,T_2$ of $\mathbb{V}$.
    \item[(iv)] $W^\perp=W^{\perp'}$, for every $\F_{q^m}$-subspace $W$ of $\mathbb{V}$.
    \item[(v)] Let $W$ and $U$ be an $\F_{q^m}$-subspace and an $\F_q$-subspace of $\fqm^k$ of dimension $s$ and $t$, respectively. Then
    \begin{equation}\label{eq:dualweight} \dim_{\F_q}(U^{\perp'}\cap W^{\perp'})-\dim_{\F_q}(U\cap W)=mk-t-sm. \end{equation}
    \item[(vi)] Let $\sigma$, $\sigma_1$ be non-degenerate reflexive sesquilinear forms over $\mathbb{V}$ and define $\smash{\sigma^\prime}$, $\smash{\sigma_1^\prime}$, $\perp$, $\perp_1$, $\perp'$, and $\smash{\perp_1'}$ as above. Then there exists an invertible $\F_{q^m}$-semilinear map $f$ such that $\smash{f(U^{\perp'})=U^{\perp_1'}}$, i.e. $\smash{U^{\perp'}}$ and $\smash{U^{\perp_1'}}$ are $\Gamma\mathrm{L}(k,q^m)$-equivalent. 
\end{itemize}
\end{proposition}

We will also need the action of the duality restricted on subspaces of the entire space, for that we will use the notation $\perp^*$ (or $\perp'^*$) to distinguish from the duality acting on the entire space.

We refer to \cite{lavrauw2015field} and \cite{polverino2010linear} for comprehensive references on linear sets.

\subsection{Rank-metric codes and their geometry}

Rank-metric codes are error-correcting codes where the distance between codewords is defined using the rank metric, given by the rank of the difference of two matrices over a finite field. These codes are particularly effective in settings where data are represented as matrices, such as in network coding and distributed storage; see \cite{bartz2022rank,gorla2021rank}.
We will describe rank-metric codes in the vectorial framework.
The \emph{rank (weight)} $w(v)$ of a vector $v=(v_1,\ldots,v_n) \in \F_{q^m}^n$ is defined as $w(v)=\dim_{\fq} (\langle v_1,\ldots, v_n\rangle_{\fq})$. 

A \emph{(linear vector) rank-metric code} $\C $ is an $\F_{q^m}$-subspace of $\F_{q^m}^n$ endowed with the rank distance, where such a distance is defined as $d(x,y)=w(x-y)$, where $x, y \in \F_{q^m}^n$. 
A $k$-dimensional rank-metric code $\C$ in $\F_{q^m}^n$ with minimum distance $d$ is also referred to as an $[n, k, d]_{q^m/q}$ \emph{code} (or simply an $[n, k]_{q^m/q}$ \emph{code}).
For other notation and terminologies we refer to \cite{alfarano2022linear,Randrianarisoa2020ageometric,sheekey2016new}.

Given a rank-metric code $\C$ and a non-negative integer $i$, we define \[
A_i=A_i (\C):=\lvert \{c \in \C \colon w(c)=i\} \rvert.\]
The sequence $(A_i)_{i\geq 0}$ is said to be the \emph{weight distribution} of $\C$.


By the classification of $\F_{q^m}$-linear isometry of $\F_{q^m}^n$ (see \cite{berger2003isometries}), we say that two $[n,k]_{q^m/q}$ codes $\C,\C$ are \emph{(linearly) equivalent} if and only if there exists a matrix $A \in \mathrm{GL}(n,q)$ such that
$\C'=\C A=\{vA : v \in \C\}$. 


Similarly to the Hamming metric, one can prove a Singleton-like bound for rank-metric codes. 

\begin{theorem}(see \cite{de78}) \label{th:singletonrank}
    Let $\C$ be an $[n,k,d]_{q^m/q}$ code.
Then 
\begin{equation}\label{eq:boundgen}
mk \leq \max\{m,n\}(\min\{m,n\}-d+1).\end{equation}
\end{theorem}

An $[n,k,d]_{q^m/q}$ code is called \emph{Maximum Rank Distance code} (or in a short form \emph{MRD code}) if its parameters reach the bound \eqref{eq:boundgen}.

A key point in the theory of rank-metric codes has been the geometric viewpoint via systems. Indeed, this points out a connection between rank-metric codes and linear sets. We recall this connection.

\begin{theorem}(see \cite{Randrianarisoa2020ageometric}) \label{th:connection}
Let $\C$ be a non-degenerate\footnote{i.e. the dimension of the $\fq$-columnspan of any generator matrix of $\C$ is $n$.} $[n,k,d]_{q^m/q}$ code and let $G$ be a generator matrix.
Let $U \subseteq \F_{q^m}^k$ be the $\F_q$-span of the columns of $G$.
The rank weight of an element $x G \in \C$, with $x \in \F_{q^m}^k$ is
\begin{equation}\label{eq:relweight}
w(x G) = n - \dim_{\fq}(U \cap x^{\perp}),\end{equation}
where $x^{\perp}=\{y \in \F_{q^m}^k \colon x \cdot y=0\}.$ In particular,
\begin{equation} \label{eq:distancedesign}
d=n - \max\left\{ \dim_{\fq}(U \cap H)  \colon H\mbox{ is an } \F_{q^m}\mbox{-hyperplane of }\F_{q^m}^k  \right\}.
\end{equation}
\end{theorem}

In other words, any non-degenerate code can be studied via an associated system; an $[n,k,d]_{q^m/q}$ \emph{system} $U$ is an $\F_q$-subspace of $\F_{q^m}^k$ of dimension $n$, such that
$ \langle U \rangle_{\F_{q^m}}=\F_{q^m}^k$ and
$$ d=n-\max\left\{\dim_{\F_q}(U\cap H) \mid H \textnormal{ is an $\F_{q^m}$-hyperplane of }\F_{q^m}^k\right\}.$$

Actually, the above result allows us to give a one-to-one correspondence between equivalence classes of non-degenerate $[n,k,d]_{q^m/q}$ codes and equivalence classes (via the action of $\mathrm{GL}(k,q^m)$) of $[n,k,d]_{q^m/q}$ systems, see \cite{Randrianarisoa2020ageometric}.
The system $U$ and the code $\C$ and in Theorem \ref{th:connection} are said to be \emph{associated}.


Moreover, we point out that the semilinear inequivalence on linear rank-metric codes can be read also on the associated systems via the action of $\mathrm{\Gamma L}(k,q^m)$ on the $\fq$-subspaces of $\F_{q^m}^n$; see \cite{sheekeysurvey} and \cite{sheekeyVdV}.

\section{Generalities on clubs}\label{sec:genclubs}

In this section, we will review the known results and constructions of clubs, and describe some basic properties.

Up to now, very few examples of $i$-clubs are known, which have been found and summarized in \cite{de2016linear}. The classical example is given by the one defined by the trace function, i.e. $\mathrm{Tr}_{q^m/q}(x)=x+\ldots+x^{q^{m-1}}$. Indeed, the linear set $L_{\mathrm{Tr}_{q^m/q}(x)}=\{\la (x,\mathrm{Tr}_{q^m/q}(x))\ra_{\fqm} \colon x \in \fqm^*\}$ is an example of $(m-1)$-club in $\PG(1,q^m)$ and in \cite[Theorem 3.7]{csajbok2018classes}, it has been proved that every $(m-1)$-club is $\mathrm{PGL}(2,q^m)$-equivalent to $L_{\mathrm{Tr}_{q^m/q}(x)}$.

A further important example is the following, which extends the previous one. Let $m=\ell n$, $i=n(\ell-1)$, $\gcd(s,n)=1$ and $\sigma\colon x \in \fqm \mapsto x^{q^s} \in \fqm$. Then the linear set $L_T=\{\la (x,T(x))\ra_{\fqm} \colon x \in \fqm^*\}$ where 
\begin{equation}\label{eq:firstKMarc}
 T(x)=\left(\mathrm{Tr}_{q^{\ell n}/q^{n}}\circ\sigma\right)(x) \in \fqm[x]
\end{equation}
is an $i$-club in $\mathrm{PG}(1,q^{m})$ (see \cite[Theorem 3.3]{de2016linear}).

In \cite{de2016linear}, two additional families of club linear sets were detected; see \cite[Lemma 2.12]{de2016linear} and \cite[Lemma 3.6]{de2016linear}.
Recently, a complete classification of $(m-2)$-clubs in $\PG(1, q^m)$, along with their polynomial representations, was provided in \cite{napolitano2022clubs}.
We also note that polynomials defining $i$-clubs were also given in \cite[Corollary 5.5]{bartoli2022r} for the case $m = 4$ (see also \cite{csajbok2018maximum}) and in \cite[Corollary 6.3]{bartoli2022r}.



We are now going to describe some properties of clubs which will be useful for the results of the next sections. We start with the property that if $L_U$ is an $i$-club with $i>1$ then it contains an $(i-1)$-club.

\begin{lemma}\label{lemma:subclub}
    Let $L_U$ be an $i$-club of rank $n$ in $\PG(k-1,q^m)$. There exists an $(i-1)$-club $L_{U'}$ of rank $n-1$, such that $L_{U'}\subseteq L_U$. 
    In particular, there exists an $(i-j)$-club $L_{U'}$ of rank $n-j$ for any $j\in \{1,\ldots,i-1\}$.
\end{lemma}
\begin{proof}
    Let $\langle v \rangle_{\fqm} \in L_U$ such that $w_{L_U}(\langle v \rangle_{\fqm})=i$. Consider an $\fq$-basis $B=\{v_1,\ldots, v_i\}$ of $\langle v \rangle_{\fqm}\cap U$. We can extend this basis to an $\fq$-basis of $U$, namely $B \cup \{w_{i+1},\cdots, w_{n}\}$. Define $U'=\langle v_2, \cdots, v_i, w_{i+1},\cdots w_n\rangle_{\fq}$. It is easy to see that $L_{U'}$ is an $(i-1)$-club of rank $n-1$ such that $L_{U'}\subseteq L_U$.
    Using finite induction, we obtain the second part of the assertion.
\end{proof}

The following result gives a method to get examples of clubs in projective spaces via a direct sum of an $i$-club in a smaller space and a scattered space.

\begin{lemma}\label{le:iclubandscatered}
    Let $T_1$ and $T_2$ be two subspaces of $\mathrm{PG}(k-1,q^m)$ such that $T_1\oplus T_2=\mathrm{PG}(k-1,q^m)$.
    Let $L_{U_1}$ and $L_{U_2}$ be two $\fq$-linear sets such that $L_{U_i}\subseteq T_i$ for any $i \in \{1,2\}$ having rank $n_1$ and $n_2$, respectively. Suppose that $L_{U_1}$ is an $i$-club and $L_{U_2}$ is a scattered $\fq$-linear set. We have that $L_U$ is an $i$-club of rank $n_1+n_2$ in $\mathrm{PG}(k-1,q^m)$, where
    $$U=U_1\oplus U_2.$$
\end{lemma}
\begin{proof}
     Clearly, $L_U$ is an $\fq$-linear set in $\mathrm{PG}(k-1,q^m)$ of rank $n_1+n_2$. We are only required to show that $L_U$ is an $i$-club. 
    Write $T_1=\mathrm{PG}(W_1,\fqm)$ and $T_2=\mathrm{PG}(W_2,\fqm)$, where $W_1$ and $W_2$ are $\fqm$-subspaces of $\mathbb{V}$ with $W_1\oplus W_2$. Let $P=\langle u \rangle_{\fqm}$ be the only point of $L_{U_1}$ having weight greater than one. 
    Consider $Q=\langle w \rangle_{\F_{q^m}} \in L_{U}$, then we have that $w=w_1 + w_2$ for some $w_1 \in W_1$ and $w_2 \in W_2$. Suppose that $w_{L_{U}}(Q)>1$, we have that $\rho w \in \langle w \rangle_{\F_{q^m}}\cap U$ for every $\rho \in S$, where $S$ is an $\fq$-subspace of $\fqm$ of dimension larger than one.
    Since $U=U_1\oplus U_2$, $W_1\oplus W_2 =\mathbb{V}$ and since
\[\rho w=\rho w_1+\rho w_2,\] we have $\rho w_1 \in U_1$ and $\rho w_2 \in U_2$ for every $\rho \in S$. Since $L_{U_2}$ is scattered, we have $w_2=0$ and since $P$ is the only point of $L_{U_1}$ with weight greater than one, we prove that $P=Q$ and $w_{L_U}(P)=w_{L_U}(Q)=i$.
\end{proof}

\section{Bounds on the rank of an $i$-club}\label{sec:bound}

In this section, we aim to provide bounds on the rank of an $i$-club in $\mathrm{PG}(k-1,q^m)$. To achieve this, we will use a rank-metric code associated with the dual of the considered club. In this way, the weight distribution of the code is entirely determined by the weight distribution of the points of the club. We will then use MacWilliams identities to demonstrate the desired bounds.
First of all, we observe that we can consider a rank-metric code associated with the dual of an $i$-club due to the following property.

\begin{proposition}\label{prop:dualconsidercode}
    Let $L_U$ be an $i$-club in PG$(k-1,q^m)$ with $i<m$. We have $\langle U^{\perp'}\rangle_{\F_{q^m}}=\mathbb{V}$. 
\end{proposition}
\begin{proof}
    Suppose that there exists a hyperplane $H$ of $\mathbb{V}$ containing $U^{\perp'}$. By (iii) of Proposition \ref{prop:dualityproperties}, we have that $H^\perp \subseteq U$.
    Since $\dim_{\F_q}(H^\perp)=m$ (by (ii) of Proposition \ref{prop:dualityproperties}) and $i<m$, we get that in $L_U$ there is a point of weight $m$, a contradiction.
\end{proof}

\begin{remark}
Note that the above proposition is not true when considering $m$-clubs. Indeed, in this case, $U^{\perp'}$ is contained in the dual of the point having weight $m$ in $L_U$. Therefore, in this case, we cannot associate a rank-metric code to its dual.
\end{remark}

Hence, we now consider a rank-metric code associated with the dual of an $i$-club. In the following result, we determine the parameters and the weight distribution of this code.

\begin{proposition}\label{prop:codeparam}
    Let $L_U$ be an $i$-club of rank $n$ in PG$(k-1,q^m)$ with $i<m$ and let $\C$ be a rank-metric code associated with $U^{\perp'}$. We have that $\C$ is an $[km-n,k,m-i]_{q^m/q}$. Moreover, if $i>1$ then its weight distribution is
    \begin{itemize}
        \item $A_0=1$;
        \item $A_{m-i}=q^m-1$;
        \item $A_{m-1}=(q^m-1)(q^n+\ldots+q^i)$;
        \item $A_m=q^{mk}-1-A_{m-1}-A_{m-i}$;
        \item $A_j=0$ for any $j \notin\{0,m-i,m-1,m\}$;
    \end{itemize}
    and if $i=1$ then its weight distribution is
    \begin{itemize}
        \item $A_0=1$;
        \item $A_{m-1}=(q^m-1)(q^n+\ldots+q+1)$;
        \item $A_m=q^{mk}-1-A_{m-1}$;
        \item $A_j=0$ for any $j \notin\{0,m-1,m\}$.
    \end{itemize}
\end{proposition}
\begin{proof}
    By (ii) of Proposition \ref{prop:dualityproperties}, we have that $\dim_{\fq}(U^{\perp'})=km-n$ and so the length of $\C$ is $km-n$ and clearly $\dim_{\F_{q^m}}(\C)=k$ by Proposition \ref{prop:dualconsidercode}. 
    From \eqref{eq:relweight}, we need to determine the possible weights of the hyperplanes with respect to $U^{\perp'}$ in order to determine the weights of $\C$.
    Observe that from (v) of Proposition \ref{prop:dualityproperties} we have that
    \[ \dim_{\F_q}(U^{\perp'}\cap H)=(k-1)m-n+\dim_{\fq}(U\cap H^\perp), \]
    for any $\F_{q^m}$-subspace of $\F_{q^m}^k$ of dimension $k-1$.
    Note that the number of hyperplanes having weight $(k-1)m-n+\dim_{\fq}(U\cap H^\perp)$, with $j \in \{0,1,i\}$, is equal to $N_j(U)$.
    Now, let $G$ be a generator matrix of $\C$ whose $\fq$-span of the columns corresponds to $U^\perp$.
    Using Theorem \ref{th:connection}, we have that for any $x \in \F_{q^m}^k$
    \[w(xG)=mk-n-\dim_{\fq}(U^\perp \cap x^\perp).\]
    By the above considerations we have that the nonzero weight of $\C$ are $m-i,m-1$ and $m$.
    The assertion is proved by noting that proportional codewords are associated with the same hyperplane (see Theorem \ref{th:connection}) and therefore $A_{m-j}=(q^m-1)N_j(U)$ for every $j$.
\end{proof}

\begin{remark}
    Note that the code we are considering is also called as the geometric dual of the code associated with the subspace $U$; see \cite{BoreZull0}.
    Also, using the correspondence between MRD codes and scattered subspaces proved in \cite{zini2021scattered}, we have that the code $\C$ of the above proposition is MRD if and only if $U$ is scattered and $n=mk/2$ if $k>2$ (see \cite{zini2021scattered,marino2023evasive}).
\end{remark}

First, let us prove a simple bound on the existence of an $i$-club.

\begin{proposition}\label{prop:genbound}
    Let $U$ be an $i$-club of rank $n$ in PG$(k-1,q^m)$ with $i<m$. We have $n\leq (k-1)m$.
\end{proposition}
\begin{proof}
    By contradiction, assume that $n>(k-1)m$. By Proposition \ref{prop:linset=entirespace}, it immediately follows that $L_U=\mathrm{PG}(k-1,q^m)$.
    Since $L_U$ is an $i$-club, it has only one point of weight greater than one and so the number of points of weight one is $N_1(U)=(q^{mk}-1)/(q^m-1)-1$.
    Whereas, by Proposition \ref{prop:size}, we have that
    \[ N_1(U)=q^{n-1}+\ldots+q^i. \]
    Since $(q^{mk}-1)/(q^m-1)-1$ is congruent to $0$ modulo $q^m$ and $q^{n-1}+\ldots+q^i$ is not, indeed if $n\leq m$ this is clearly true, whereas if $n>m$ then $q^{n-1}+\ldots+q^i \equiv q^{m-1}+\ldots+q^i$ which cannot be zero modulo $q^m$. Therefore, these two values are distinct and we have a contradiction.
\end{proof}

We can apply the Singleton bound on the code of the above proposition in order to get a first bound.

\begin{proposition}
    Let $L_U$ be an $i$-club of rank $n\leq (k-1)m$ in PG$(k-1,q^m)$ with $i<m$. We have
    \[ n\leq \frac{kmi}{i+1}. \]
\end{proposition}
\begin{proof}
    Consider a rank-metric code $\C$ associated with $U^\perp$ (along the same lines as of the proof of Proposition \ref{prop:codeparam}). Note that 
    \[ \max\{ m,km-n \}=km-n\,\,\text{and}\,\,\min\{ m,km-n \}=m, \]
    as $n\leq (k-1)m$.
    Therefore, thanks to Proposition \ref{prop:codeparam}, we know that the minimum distance of $\C$ is $m-i$ and the Singleton bound in Theorem \ref{th:singletonrank} reads as follows
    \[ mk\leq (km-n)(i+1), \]
    from which we derive the assertion.
\end{proof}

Our aim is to improve the above bound with the aid of the MacWilliams identities, which relate the weight distribution of a code $\mathcal{C}$ to that of its dual. Recall that the \emph{dual code} of an \( [n,k]_{q^m/q} \) code \( \mathcal{C} \) is defined as
\[
\mathcal{C}^\perp = \left\{ (d_1,\ldots,d_n) \in \mathbb{F}_{q^m}^n \, \colon \, \sum_{i=1}^n c_i d_i = 0 \text{ for every } (c_1,\ldots,c_n) \in \mathcal{C} \right\}.
\]
The identities take into account the entire weight distribution of the code, not just the minimum weight. In what follows, for two integers $s$ and $t$, the $q$-binomial coefficient of $s$ and $t$ is denoted and defined by
\[
\qbin{s}{t}{q} := 
\begin{cases}
    0 & \text{if } s < 0,\, t < 0, \text{ or } t > s, \\
    1 & \text{if } t = 0 \text{ and } s \geq 0, \\
    \prod\limits_{i=0}^{t-1} \frac{q^{s-i} - 1}{q^{i+1} - 1} & \text{otherwise}.
\end{cases}
\]

We now rephrase the MacWilliams identities in our notation.

\begin{theorem}(see \cite[Theorem 31]{ravagnani2016rank})\label{thm:MWident}
    Let $\C$ be an $[N,k]_{q^m/q}$ code. Let $(A_i)_{i \geq 0}$ and $(B_j)_{j \geq 0}$ be the rank distributions of $\C$ and $\C^\perp$, respectively. For any integer $0\leq \nu \leq m$ we have
    \[ \sum_{i=0}^{m-\nu} A_i \qbin{m-i}{\nu}{q}=\frac{|\C|}{q^{N\nu}}\sum_{j=0}^\nu B_j\qbin{m-j}{\nu-j}{q}. \]
\end{theorem}

We are know ready to use MacWilliams identities to show the desired upper bound.

\begin{theorem}\label{th:mainBound}
    Let $L_U$ be an $i$-club of rank $n$ in PG$(k-1,q^m)$ with $2 \leq i<m$. We have
    \[ n\leq  
    \begin{cases}
        \frac{mk}2 & \text{if }\,\, i\leq m/2,\\
        \frac{m(k-1)}2+i & \text{if }\,\, i\geq m/2.
    \end{cases}
    \]
\end{theorem}
\begin{proof}
   Consider $\C$ be an $[km-n,k]_{q^m/q}$ code associated with $U^\perp$ (as in Proposition \ref{prop:codeparam}). 
    Clearly $B_0=1$.
    By \cite[Proposition 3.2]{alfarano2022linear}, the minimum distance of $\C^\perp$ is at least two as $\C$ is non-degenerate. Hence, we have that $B_1=0$. Let us now determine $B_2$ via Theorem \ref{thm:MWident} with $\nu=2$.
    In this case we have
    \[ \sum_{j=0}^{m-2}A_j \qbin{m-j}{2}{q}=\frac{q^{mk}}{q^{2(km-n)}}\left( \qbin{m}{2}{q}+B_2\right), \]
    from which, by using Proposition \ref{prop:codeparam}, we derive 
    \begin{equation}\label{eq:B2} B_2=q^{km-2n}\left( (q^m-1)\qbin{i}{2}{q}+\qbin{m}{2}{q} \right)-\qbin{m}{2}{q}. \end{equation}
    Clearly, if the code $\C$ exists then $B_2\geq 0$. Therefore, from \eqref{eq:B2}, we have
    \begin{equation}\label{eq:supercond} 
    q^{km-2n}(q^i-1)(q^{i-1}-1)+(q^{km-2n}-1)(q^{m-1}-1)\geq 0. 
    \end{equation}
    Hence, if we prove that \eqref{eq:supercond} is not satisfied, then the code $\C$ cannot exist. We divide the analysis according to whether $i\leq m/2$ or $i> m/2$.\\
    \textbf{Case 1:} $i\leq m/2$.\\
    Assume that $n\geq mk/2+1$. In this case, we have that 
    \[ q^{km-2n}(q^i-1)(q^{i-1}-1)+(q^{km-2n}-1)(q^{m-1}-1) \leq q^{-2}(q^{m/2}-1)(q^{m/2-1}-1)+(q^{-2}-1)(q^{m-1}-1). \]
    It is easy to see that $q^{-2}(q^{m/2}-1)(q^{m/2-1}-1)+(q^{-2}-1)(q^{m-1}-1)<0$, yielding a contradiction to \eqref{eq:supercond}. Hence, $n\leq mk/2$. \\  
    \textbf{Case 2:} $i\geq m/2$.\\
    Assume that $n\geq m(k-1)/2+i+1$. In this case, we have that
    \[  q^{km-2n}(q^i-1)(q^{i-1}-1)+(q^{km-2n}-1)(q^{m-1}-1)\leq  q^{m-2i-2}(q^i-1)(q^{i-1}-1)+(q^{m-2i-2}-1)(q^{m-1}-1), \]
    using that $i\leq m-1$ we also get
    \[  q^{km-2n}(q^i-1)(q^{i-1}-1)+(q^{km-2n}-1)(q^{m-1}-1)\leq  q^{-m}(q^{m-1}-1)(q^{m-2}-1)+(q^{-m}-1)(q^{m-1}-1), \]
    and, as before, we can check that 
    \[q^{-m}(q^{m-1}-1)(q^{m-2}-1)+(q^{-m}-1)(q^{m-1}-1)<0,\]
    so that we get again a contradiction to \eqref{eq:supercond}.
\end{proof}

When $k=2$ the best bound we can get is the one provided in Proposition \ref{prop:genbound} when $i<m$. We resume the case in which $k=2$ in the following proposition.

\begin{proposition}
    Let $L_U$ be an $i$-club of rank $n$ in PG$(1,q^m)$. We have
    \[ n\leq  
    \begin{cases}
        m & \text{if }\,\, i\leq m-1,\\
        m+1 & \text{if }\,\, i=m.
    \end{cases}
    \]
\end{proposition}
\begin{proof}
    The first part follows from Proposition \ref{prop:genbound}. When $i=m$ we have that $n\leq m+1$, otherwise every point will have weight at least two. 
\end{proof}


In Theorem \ref{th:mainBound}, we have assumed that $i<m$. We now provide a bound on the rank of $m$-clubs.

\begin{proposition}
    Let $L_U$ be an $m$-club of rank $n$ in PG$(k-1,q^m)$ and let $k\geq 3$. We have
    \[ n\leq  \frac{m(k-1)}2+m.\]
\end{proposition}
\begin{proof}
    First we observe that $n\leq (k-1)m+1$, otherwise all the points of $L_U$ would have weight larger than or equal to two. By Lemma \ref{lemma:subclub}, we have that there exists a subspace $U'$ of $U$ such that $L_{U'}$ is an $(m-1)$-club in PG$(k-1,q^m)$ of rank $n-1\leq m(k-1)$. Therefore, we can apply Theorem \ref{th:mainBound} to $L_{U'}$ obtaining that $n-1\leq m(k-1)/2+m-1$ from which the assertion follows.
\end{proof}

Combining the results of this section, we have proved the following.

\begin{corollary}\label{cor:boundiclub}
    Let $L_U$ be an $i$-club of rank $n$ in PG$(k-1,q^m)$ with $2 \leq i\leq m$. We have
    \[ n\leq  
    \begin{cases}
        m+1 & \text{if }\,\, i= m\,\, \text{and }\, k=2,\\
        \frac{mk}2 & \text{if either }\,\, i\leq m/2 \,\, \text{or } k=2 \ \mbox{and} \ i \leq m-1,\\
        \frac{m(k-1)}2+i & \text{if }\,\, m/2\leq i\leq m \ \mbox{ and }\ k>2.
    \end{cases}
    \]
\end{corollary}

\begin{remark}
    We note the above bounds are not always sharp. Indeed, in the case in which $i=k=2$ and $m=5$, De Boeck and Van de Voorde in \cite{de2016linear} have proved that there do not exist $2$-clubs of rank $5$ in $\PG(1,q^5)$.
    As we will see in the next section, in the case in which $i\geq m/2$ and $k>3$ the bounds are tight. 
\end{remark}
 
\section{Constructions}\label{sec:constr}

In this section, we describe several explicit constructions of $i$‑clubs whose rank attains equality in Corollary \ref{cor:boundiclub} for the case $i \geq m/2$. We begin with what we call the \emph{cone construction}, which establishes the existence of an $i$‑club with maximum rank whenever $i\ge m/2$ and $m(k-1)$ is even. Next, we introduce two further constructions — one for odd $k$ and one for even $k$ — and demonstrate that neither is $\Gamma\mathrm{L}(k,q^m)$‑equivalent to the cone construction. Our non‐equivalence proofs rely on a detailed analysis of the weight‐hyperplane distribution in each case.

\subsection{Cone construction}

In this section, we focus on providing an example of an $i$-club of maximum rank, i.e.\ reaching the equality in Corollary \ref{cor:boundiclub}, for any $i\geq m/2$. 
The main idea is to construct it by considering a maximum scattered linear set in a hyperplane and a point external to the hyperplane.

As a consequence of Theorem \ref{th:existencescattered}, we obtain a first existence results for $i$-clubs in projective spaces whenever $i\geq m/2$.

\begin{proposition}[Cone construction]\label{prop:coneconstruct}
    Let $k$ and $m$ be two positive integers such that $(k-1)m$ is even.
    Let $H=\PG(W,\fqm)$ be a hyperplane of $\mathrm{PG}(k-1,q^m)$. Let $L_{U'}$ be a maximum scattered $\fq$-linear set in $H$.
    Let $x\in \mathbb{F}_{q^m}^k\setminus\{0\}$, such that $\langle x\rangle_{\mathbb{F}_{q^m}}\cap  W= \{0\}$. If $U''$ is an $i$-dimensional $\fq$-subspace of $\langle x\rangle_{\mathbb{F}_{q^m}}$, for some positive integer $i\geq m/2$, and 
    $$U=U'+U'',$$
    then $L_U$ is an $i$-club of rank $n=\frac{(k-1)m}{2}+i$.
    In particular, for every $k$ and $m$ positive integers such that $(k-1)m$ is even, and $q$ a prime power, there exists an $i$-club of maximum rank in $\mathrm{PG}(k-1,q^m)$.
\end{proposition}
\begin{proof}
     The statement is a consequence of Lemma \ref{le:iclubandscatered} and from the existence of maximum scattered linear sets $H$, due to Theorem \ref{th:existencescattered}.\\
\end{proof}

Therefore, we have that the bound of Corollary \ref{cor:boundiclub} is tight when $i\geq m/2$.

\begin{corollary}
    For any positive integers $k,m$ such that $(k-1)m$ is even and any $i\geq m/2$, the bound in Corollary \ref{cor:boundiclub} is tight.
\end{corollary}

Our aim is also to describe the possible weights of hyperplanes with respect to the linear set constructed in the above proposition. The construction is a \emph{direct sum}, which allows us to exploit its structure by extending the following result.

\begin{proposition} \label{prop:weightbyintersectioninfinityind}(see \cite[Proposition 3.2]{AMSZ23})
Let $H=\PG(W,\fqm)$ be a hyperplane of $\mathrm{PG}(k-1,q^m)$. Let $L_{U'}$ be an $\fq$-linear set in $H$ and let $P =\langle x\rangle_{\mathbb{F}_{q^m}}$ be a point not in $H$.
Let $U''$ be a $1$-dimensional $\fq$-subspace of $\langle x\rangle_{\mathbb{F}_{q^m}}$ and  $U:=U'+U''$.
Consider $\Omega=\PG(T,\F_{q^m})$ a subspace of $\PG(k-1,q^m)$.
We have that
\[ w_{L_U}(\Omega) \in \{w_{L_{U'}}(H\cap \Omega),w_{L_{U'}}(H\cap \Omega)+1\}. \]
\end{proposition}

In the next, we extend the above result by replacing $U''$ with a subspace of higher dimension.

\begin{proposition} \label{prop:weightbyintersectioninfinity}
Let $H=\PG(W,\fqm)$ be a hyperplane of $\mathrm{PG}(k-1,q^m)$. Let $L_{U'}$ be an $\fq$-linear set in $H$ and let $P =\langle x\rangle_{\mathbb{F}_{q^m}}$ be a point not in $H$.
Let $U''$ be an $i$-dimensional $\fq$-subspace of $\langle x\rangle_{\mathbb{F}_{q^m}}$ and  $U:=U'+U''$.
Consider $\Omega=\PG(T,\F_{q^m})$ a subspace of $\PG(k-1,q^m)$.
We have that
\[ w_{L_U}(\Omega) \in \{w_{L_{U'}}(H\cap \Omega),\ldots,w_{L_{U'}}(H\cap \Omega)+i\}. \]
\end{proposition}
\begin{proof}
    Consider $\overline{U}=U'+\tilde{U}$, where $\tilde{U}$ is a one-dimensional $\fq$-subspace of $U''$ and observe that
    \begin{equation}\label{eq:addi-1atmost}
        w_{L_U}(\Omega)\leq w_{L_{\overline{U}}}(\Omega)+i-1,
    \end{equation}
    as $\dim_{\fq}(U)=\dim_{\fq}(\overline{U})+i-1$.
    By Proposition \ref{prop:weightbyintersectioninfinityind}, we know that 
    \[w_{L_{\overline{U}}}(\Omega) \in \{w_{L_{U'}}(H\cap \Omega),w_{L_{U'}}(H\cap \Omega)+1\},\]
    and so, by \eqref{eq:addi-1atmost} we have that $w_{L_U}(\Omega) \leq w_{L_{U'}}(H\cap \Omega)+i$.
\end{proof}

In order to compute the possible weights of the hyperplanes with respect to the above constructed $i$-club we need to know which are the weights of the hyperplane with respect to a maximum scattered linear set.

\begin{theorem}{(see \cite[Theorem 4.2]{blokhuis2000scattered})}
 \label{th:inter}
If $L_U$ is a maximum scattered $\F_q$-linear set of PG$(k-1,q^m)$, then for any hyperplane $H$ of PG$(k-1,q^m)$ we have
\[
w_{L_{U}}(H) \in \left\{  \frac{km}{2}-m,  \frac{km}{2}-m+1\right\}.
\]
\end{theorem}

Combining Proposition \ref{prop:weightbyintersectioninfinity} and Theorem \ref{th:inter}, we find the possible weights of the hyperplanes with respect to the $i$-club constructed in Proposition \ref{prop:coneconstruct}.

\begin{theorem}\label{th:coneconstruct}
Let $k$ and $m$ be two positive integers such that $(k-1)m$ is even.
Let $H=\PG(W,\fqm)$ be a hyperplane of $\mathrm{PG}(k-1,q^m)$. Let $L_{U'}$ be a maximum scattered $\fq$-linear set in $H$.
Let $x\in \mathbb{F}_{q^m}^k$, such that $\langle x\rangle_{\mathbb{F}_{q^m}}\cap  W= \{0\}$. Let $U''$ be an $i$-dimensional $\fq$-subspace of $\langle x\rangle_{\mathbb{F}_{q^m}}$, for some positive integer $i\geq m/2$, and let $U:=U'+U''$.
For any hyperplane $\Omega$ of $\mathrm{PG}(k-1,q^m)$, we have that
\[ w_{L_U}(\Omega)\in \left\{ \frac{m(k-1)}2,\frac{m(k-2)}2, \frac{m(k-2)}2+1,\ldots, \frac{m(k-2)}2+i+1 \right\}. \]
In particular, there exists an hyperplane having weight $\frac{m(k-1)}2$.
\end{theorem}
\begin{proof}
Consider a hyperplane $\Omega$ of $\mathrm{PG}(k-1,q^m)$.
We have that either $\Omega=H$ or $\Omega \cap H$ is a hyperplane of $H$.
In the first case, it is easy to see that $\Omega \cap L_{U}=L_{U'}$ and so \[w_{L_U}(\Omega)=\mathrm{Rank}(L_{U'})=\frac{m(k-1)}2.\]
In the second case, consider $\Omega'=\Omega \cap H$. Recall that $L_{U'}$ is a maximum scattered $\F_q$-linear set in $H$. So, by Theorem \ref{th:inter}, \[w_{L_{U'}}(\Omega') \in \left\{  \frac{(k-1)m}{2}-m,  \frac{(k-1)m}{2}-m+1\right\}.\]
The assertion follows now from Proposition \ref{prop:weightbyintersectioninfinity} and by observing that $H$ has weight $\frac{m(k-1)}2$.
\end{proof}

A complete characterization of the possible sizes of the intersections between hyperplanes and linear sets, as described in the above theorem, can be quite intricate. In the following remark, we clarify which intersection sizes are possible.

\begin{remark}
    Let $L_U$ be as in the above theorem.
    Note that, for every hyperplane $\Omega$ of $\mathrm{PG}(k-1,q^m)$, $L_U\cap \Omega$ is either a scattered linear set or an $i$-club, depending on whether the point $P=\la x \ra_{\fqm}$ is in $\Omega$ or not. Therefore, if $w_{L_{U}}(\Omega)=j$ then
    \[|L_U \cap \Omega|= \begin{cases}
        \frac{q^j-1}{q-1} & \text{if } P \notin \Omega,\\
        q^{j-1}+\ldots+q^i+1  & \text{if } P \in \Omega.
    \end{cases}\]
\end{remark}

\subsection{Lifting construction for $k$ odd}

In this section, we present another construction of $i$-club whose rank attains equality in Corollary \ref{cor:boundiclub} for the case $i \geq m/2$.
We will also show that this new construction is not equivalent to the cone construction presented in the previous section, by comparing the weights of the hyperplanes.

Throughout this subsection we assume that $k=2s+1>1$, for some positive integer $s$.

\begin{proposition}[Lifting construction] \label{con:conekodd}
    Let $k$ and $m$ be two positive integers such that $k=2s+1$, for some positive integer $s$.
    Consider
    $$U=\{(x_1+\zeta,x_1^q,x_1^{q^2}, x_2, x_2^{q}, \ldots, x_s,x_s^{q} )\mid x_1,\ldots,x_s\in \mathbb{F}_{q^m}, \zeta\in S\},$$
    for an $\fq$-subspace $S$ of $\fqm$ with dimension $i$. We have that $L_U$ is an $i$-club in $\mathrm{PG}(k-1,q^m)$ of rank $m(k-1)/2+i$.
\end{proposition}
\begin{proof}
    Let $v \in U\setminus \{0\}$ and let us determine $\la v \ra_{\fqm}\cap U$. Our aim is to show that if $\dim_{\fq}(\la v \ra_{\fqm}\cap U)>1$ then $\la v \ra_{\fqm}=\la (1,0,\ldots,0)\ra_{\fqm}$.
    Since $v \in U$, we have that there exist  $x_1,\ldots,x_s\in \mathbb{F}_{q^m}, \zeta\in S$ such that
    $v=(x_1+\zeta,x_1^q,x_1^{q^2}, x_2, x_2^{q}, \ldots, x_s,x_s^{q})$. Observe that 
    \[ \dim_{\fq}(\la v \ra_{\fqm}\cap U)=\dim_{\fq}(\{ \rho \in \fqm \colon \rho v \in U \}), \]
    and so let us study this latter subspace.
    Let $\rho \in \fqm$ and assume that 
    \[ \rho (x_1+\zeta,x_1^q,x_1^{q^2}, x_2, x_2^{q}, \ldots, x_s,x_s^{q})=(y_1+\eta,y_1^q,y_1^{q^2}, y_2, y_2^{q}, \ldots, y_s,y_s^{q}), \]
    for some $y_1,\ldots,y_s\in \mathbb{F}_{q^m}, \eta\in S$, from which we immediately derive the following
    \[ 
    \left\{
\begin{array}{llll}
     \rho(x_1+\zeta)=y_1+\eta,  \\
     \rho x_1^q=y_1^q,\\
     \rho x_1^{q^2}=y_1^{q^2},\\
     \rho x_2=y_2,\\
     \vdots\\
     \rho x_s^q=y_s^q.
\end{array}
    \right.
    \]
    Note that if there exists $i \in \{1,\ldots,s\}$ with $i\ne 1$ such that $x_i\ne 0$, then we have that
    \[ \begin{cases}
        \rho x_i=y_1,\\
        \rho x_i^q=y_i^q,
    \end{cases} \]
    implying that $\rho \in \fq$ and so $\dim_{\fq}(\la v \ra_{\fqm}\cap U)=1$.
    If $x_1\ne 0$ then from the above system we derive
    \[ 
    \left\{
\begin{array}{llll}
     \rho x_1^q=y_1^q,\\
     \rho x_1^{q^2}=y_1^{q^2},
\end{array}
    \right.
    \]
    for which we get the same conclusion as before. 
    Therefore, the only case to analyze is those for which all the $x_i$'s are zero, i.e. $v=(\zeta,0,\ldots,0)$, for which we get $\dim_{\fq}(\la v \ra_{\fqm}\cap U)=i$.
\end{proof}

As a consequence, the lifting construction provided in the above result yields examples of $i$-clubs whose rank is $m(k - 1)/2 + i$, thereby achieving equality in Corollary~\ref{cor:boundiclub} for the case where $i \geq m/2$.

\begin{remark}\label{rem:directsumliftconstrkodd}
    Note that $U$ can also be written as the following direct sum
    $$U=\overline{U}_1\oplus U_2\oplus \cdots \oplus U_s,$$
     where $\overline{U}_1=\{(x_1+\zeta,x_1^q,x_1^{q^2})\mid x_1\in\fqm,\zeta \in S\}$ and $U_i:=\{(x_i^q,x_i^{q^2})\mid x_i\in\fqm\}$, for any $i\in \{2,\cdots, s\}$.
     Therefore, we could give a slightly different proof of the above theorem, first observing that $\overline{U}_1$ is an $i$-club and then using Lemma \ref{le:iclubandscatered}.
     Also, one can extend the previous construction by replacing $U_2\oplus \cdots \oplus U_s$ by a maximum scattered subspace contained in the subspace having equations $X_0=X_1=X_2=0$.
\end{remark}

Now, we show that $L_U$, where $L_U$ is as in Proposition \ref{con:conekodd}, is not contained in any hyperplane of $\mathrm{PG}(k-1,q^m)$.

\begin{proposition}
    Let $k$ and $m$ be two positive integers such that $k=2s+1$, for some positive integer $s$.
    Consider
    $$U=\{(x_1+\zeta,x_1^q,x_1^{q^2}, x_2, x_2^{q}, \ldots, x_s,x_s^{q} )\mid x_1,\ldots,x_s\in \mathbb{F}_{q^m}, \zeta\in S\},$$
    for an $\fq$-subspace $S$ of $\fqm$ with dimension $i$. We have that $L_U$ is not contained in any hyperplane of $\mathrm{PG}(k-1,q^m)$.
\end{proposition}
\begin{proof}
    The assertion is equivalent to the fact that $\la U \ra_{\fqm}=\fqm^k$.
    Observe that, if $\xi_1,\xi_2,\xi_3$ are $\fq$-linearly independent elements in $\fqm$ then
    $(\xi_1,\xi_1^q,\xi_1^{q^2}),(\xi_2,\xi_2^q,\xi_2^{q^2})$ and $(\xi_3,\xi_3^q,\xi_3^{q^2})$ are $\fqm$-linearly independent.
    The same happens if we consider $(\xi_1,\xi_1^q)$ and $(\xi_2,\xi_2^q)$.
    Therefore, the $k$ vectors
    \[ (\xi_1,\xi_1^q,\xi_1^{q^2},0,\ldots,0),(\xi_2,\xi_2^q,\xi_2^{q^2},0,\ldots,0),(\xi_3,\xi_3^q,\xi_3^{q^2},0,\ldots,0),
    (0,0,0,\xi_1,\xi_1^q,0,\ldots,0),\]
    \[(0,0,0,\xi_2,\xi_2^q,0,\ldots,0),\ldots, (0,\ldots,0,\xi_1,\xi_1^q), (0,\ldots,0,\xi_2,\xi_2^q)
    \]
    are in $U$ and are $\fqm$-linearly independent, therefore $\la U \ra_{\fqm}=\fqm^k$.
\end{proof}

Our aim now is to show that the cone construction and the lifting construction are not $\mathrm{\Gamma L}(k,q^m)$-equivalent.
In order to do so, we will study the weights of the hyperplanes on the lifting construction and then compare it with those of the cone construction.
We are going to prove some lower and upper bounds on the weights of the hyperplanes of the lifting construction by duality looking at the weight distribution of the points.
Therefore, we need to compute the dual of the subspace $U$ in Proposition \ref{con:conekodd}, by making use of Remark \ref{rem:directsumliftconstrkodd}.

\begin{lemma}\label{lem:dual}
  Consider
    \[ W=\{(x,x^q)\colon x \in \fqm\}, \]
    \[ \overline{U}=\{(x+\zeta ,x^q,x^{q^2})\colon x \in \fqm, \zeta \in S\}, \]
    where $S$ is an $\fq$-subspace of $\fqm$.
    We have that
    \[ W^{\perp'}=\{ (y,-y^q) \colon y \in \fqm \} \]
    and 
    \[ \overline{U}^{\perp'}=\{ (z,-y^q,-z^{q^2}+y^{q^2}) \colon  y \in \fqm, z \in S^\perp \}. \]
    Also, if $k=2s+1$, for some positive integer $s$, and 
    $$U=\{(x_1+\zeta,x_1^q,x_1^{q^2}, x_2, x_2^{q}, \ldots, x_s,x_s^{q} )\mid x_1,\ldots,x_s\in \mathbb{F}_{q^m}, \zeta\in S\},$$
    for an $\fq$-subspace $S$ of $\fqm$, then
    \[U^{\perp'}=\{(z,-y_1^q,-z^{q^2}+y_1^{q^2},y_2,-y_2^q,\ldots,y_s,-y_s^q) \colon y_1,y_2,\ldots,y_s \in \fqm, z \in S^{\perp^*} \}.\]
\end{lemma}
\begin{proof}
    Let 
    \[W'=\{ (y,-y^q) \colon y \in \fqm \}, \]
    and 
    \[\overline{U}'=\{ (z,-y^q,-z^{q^2}+y^{q^2}) \colon y \in \fqm , z \in S^{\perp^*} \}\]
    It is easy to show that $W'\subseteq W^{\perp'}, \ \overline{U}'\subseteq \overline{U}^{\perp'},$ $\dim_{\fq}(W')=m$ and $\dim_{\fq}(\overline{U}')=2m-\dim_{\fq}(S)$.
    Hence, the statement follows.
\end{proof}

We are now ready to provide bounds on the weights of the hyperplane with respect to the $L_U$ constructed in Proposition \ref{con:conekodd}.

\begin{theorem}\label{thm:conekoddhyper}
    Let $k$ and $m$ be two positive integers such that $k=2s+1$, for some positive integer $s$.
    Consider
    $$U=\{(x_1+\zeta,x_1^q,x_1^{q^2}, x_2, x_2^{q}, \ldots, x_s,x_s^{q} )\mid x_1,\ldots,x_s\in \mathbb{F}_{q^m}, \zeta\in S\},$$
    where $S$ is an $\fq$-subspace of $\fqm$ with dimension $i$. We have that 
    \[ w_{L_U}(H) \in \{m(k-3)/2+i,m(k-3)/2+i+1,m(k-3)/2+i+2\}, \]
    for every hyperplane $H$ in $\mathrm{PG}(k-1,q^m)$.
\end{theorem}
\begin{proof}
    Let us start by considering the dual of $U$, by Lemma \ref{lem:dual} it is equal to
    \[U^{\perp'}=\{(z,-y_1^q,-z^{q^2}+y_1^{q^2},y_2,-y_2^q,\ldots,y_s,-y_s^q) \colon  y_1,y_2,\ldots,y_s \in \fqm, z_1 \in S^{\perp'}\},\]
    and note that $\dim_{\fq}(U^{\perp'})=m(k-1)/2+m-i$.
    Observe that by \eqref{eq:dualweight} we have 
    \begin{equation}\label{eq:relconeconstrkodd}
    w_{L_U}(H)=w_{L_{U^{\perp'}}}(H^\perp)+m(k-3)/2+i.
    \end{equation}
    Therefore, we can study weight distribution of the points of $L_{U^{\perp'}}$ to get the weight distribution of the hyperplanes with respect to $L_U$.
    Let 
    \[ \la (z,-y_1^q,-z^{q^2}+y_1^{q^2},y_2,-y_2^q,\ldots,y_s,-y_s^q)  \ra_{\fqm}\in L_{U^{\perp'}}, \]
    for some $z \in S^{\perp'}$ and $y_1,y_2,\ldots,y_s \in \fqm$. As in the proof of Proposition \ref{con:conekodd}, we can find the weight of this point by finding the dimension of the subspace of $\fqm$ given by those $\rho \in \fqm$ such that
    \[ \rho(z,-y_1^q,-z^{q^2}+y_1^{q^2},y_2,-y_2^q,\ldots,y_s,-y_s^q) \in U^{\perp'}. \]
    If one of the $y_i$'s, with $i \in \{2,\ldots,s\}$ is nonzero, then we immediately get $\rho \in \fq$.
    Therefore, let us assume that $y_2=\ldots=y_s=0$ and so we have
    \[\rho(z,-y_1^q,-z^{q^2}+y_1^{q^2})=(r_1,-r_2^q,-r_1^{q^2}+r_2^{q^2}),\]
    for some $r_1 \in S^{\perp'}$ and $r_2 \in \fqm$. From this, we derive that $r_1=\rho z$, $r_2=\rho^{q^{m-1}}y_1$ and so
    \[\rho(-z^{q^2}+y_1^{q^2})=-(\rho z)^{q^2}+(\rho^{q^{m-1}}y_1)^{q^2},\]
    which can be also rewritten as follows
    \begin{equation}\label{eq:conconstrrho} 
    -\rho^{q^2}z^{q^2}+\rho^q y_1^{q^2}-\rho(-z^{q^2}+y_1^{q^2})=0.
    \end{equation}
    Since $z$ and $y_1$ cannot both be equal to zero, for fixed values of $z$ and $y_1$ there are at most $q^2$ value for $\rho$ satisfying \eqref{eq:conconstrrho}.
    Hence, \[
    w_{L_{U^{\perp'}}}(\la (z,-y_1^q,-z^{q^2}+y_1^{q^2},y_2,-y_2^q,\ldots,y_s,-y_s^q)  \ra_{\fqm})\leq 2.
    \]
    As a consequence, by \eqref{eq:relconeconstrkodd}
    \[m(k-3)/2+i\leq w_{L_U}(H)\leq m(k-3)/2+i+2,\]
    and so we have the assertion.
\end{proof}

As a consequence, we can derive the following non-equivalence of the constructions presented in Propositions \ref{prop:coneconstruct} and \ref{con:conekodd} for the case $i\leq m-3$.

\begin{corollary}
    Let $k$ and $m\geq 5$ be two positive integers such that $k=2s+1$, for some positive integer $s$.
    Let $H=\PG(W,\fqm)$ be a hyperplane of $\mathrm{PG}(k-1,q^m)$. Let $L_{U'}$ be a maximum scattered $\fq$-linear set in $H$.
    Let $x\in \mathbb{F}_{q^m}^k$, such that $\langle x\rangle_{\mathbb{F}_{q^m}}\cap  W= \{0\}$. Let $U''$ be an $i$-dimensional $\fq$-subspace of $\langle x\rangle_{\mathbb{F}_{q^m}}$, for some positive integer $i\geq m/2$, define 
    $U=U'+U''$.
    Consider
    $$\overline{U}=\{(x_1+\zeta,x_1^q,x_1^{q^2}, x_2, x_2^{q}, \ldots, x_s,x_s^{q} )\mid x_1,\ldots,x_s\in \mathbb{F}_{q^m}, \zeta\in S\},$$
    where $S$ is an $\fq$-subspace of $\fqm$ with dimension $i$.
    If $i\leq m-3$, then $U$ and $\overline{U}$ are $\mathrm{\Gamma L}(k,q^m)$-inequivalent.
\end{corollary}
\begin{proof}
    Suppose that $U$ and $\overline{U}$ are $\mathrm{\Gamma L}(k,q^m)$-equivalent, we have that 
    \begin{equation}\label{eq:weightset} \{w_{L_{U}}(\Omega) \colon \Omega \text{ hyperplane of } \mathrm{PG}(k-1,q^m)\}=\{w_{L_{\overline{U}}}(\Omega) \colon \Omega \text{ hyperplane of } \mathrm{PG}(k-1,q^m)\} \end{equation}
    By Theorem \ref{th:coneconstruct}, there exists an hyperplane $\Omega$ in $\mathrm{PG}(k-1,q^m)$ with weight $m(k-1)/2$ with respect to $L_U$. Since the possible weights for $\Omega$ with respect to $L_{\overline{U}}$ are $m(k-3)/2+i, m(k-3)/2+i+1$ and $m(k-3)/2+i+2$, by Theorem \ref{thm:conekoddhyper}, we have a contradiction to \eqref{eq:weightset}. 
\end{proof}

When $i=m-2$, we can prove that in some cases we have non-equivalent examples. Indeed, we show that for a class of examples obtained from the cone construction (with $i=m-2$) cannot contains $(m-2)$-clubs equivalent to the construction in Proposition \ref{con:conekodd}.

\begin{corollary}
    Let $k$ and $m$ be two positive integers such that $k=2s+1$, for some positive integer $s$.
    Consider
    $$U=\{(\eta,x_1,x_1^{q}, x_2, x_2^{q}, \ldots, x_s,x_s^{q} )\mid x_1,\ldots,x_s\in \mathbb{F}_{q^m}, \eta\in S\},$$
    and
    $$\overline{U}=\{(x_1+\zeta,x_1^q,x_1^{q^2}, x_2, x_2^{q}, \ldots, x_s,x_s^{q} )\mid x_1,\ldots,x_s\in \mathbb{F}_{q^m}, \zeta\in T\},$$
    where $S$ and $T$ are $\fq$-subspaces of $\fqm$ with dimension $m-2$ such that there exist no $\alpha \in \fqn$ and $\sigma \in \mathrm{Aut}(\fqm)$ such that $S=\alpha T^\sigma$.
    We have that $U$ and $\overline{U}$ are $\mathrm{\Gamma L}(k,q^m)$-inequivalent.
\end{corollary}
\begin{proof}
    By contradiction, assume that $U$ and $\overline{U}$ are $\mathrm{\Gamma L}(k,q^m)$-equivalent. Therefore, there exist a matrix $A \in \mathrm{GL}(k,q^m)$ and $\sigma \in \mathrm{Aut}(\fqm)$ such that for every $u \in \overline{U}$ there exists $v \in U$ such that
    \[ A (u^\sigma)^T=v. \]
    Denote by $A_{i,j}$ the $(i,j)$-th entry of $A$. 
    We have that for every $x_1,\ldots,x_s\in \mathbb{F}_{q^m}, \zeta\in T$ there exist $y_1,\ldots,y_s\in \mathbb{F}_{q^m}, \eta\in S$ such that
    \[
    \left\{
    \begin{array}{llll}
    \eta =A_{1,1}(x_1^\sigma+\zeta^\sigma)+A_{1,2}(x_1^\sigma)^q+\ldots+A_{1,k}(x_s^\sigma)^q,\\
    y_1=A_{2,1}(x_1^\sigma+\zeta^\sigma)+A_{2,2}(x_1^\sigma)^q+\ldots+A_{2,k}(x_s^\sigma)^q,\\
    y_1^q=A_{3,1}(x_1^\sigma+\zeta^\sigma)+A_{3,2}(x_1^\sigma)^q+\ldots+A_{3,k}(x_s^\sigma)^q,\\
    \vdots\\
    y_s=A_{2s,1}(x_1^\sigma+\zeta^\sigma)+A_{2s,2}(x_1^\sigma)^q+\ldots+A_{2s,k}(x_s^\sigma)^q,\\
    y_1^q=A_{2s+1,1}(x_1^\sigma+\zeta^\sigma)+A_{2s+1,2}(x_1^\sigma)^q+\ldots+A_{2s+1,k}(x_s^\sigma)^q.
    \end{array}
    \right.
    \]
    In particular, since $\la (1,0,\ldots,0)\ra_{\fqm}$ is the only point of weight greater than one for both the linear sets $L_U$ and $L_{\overline{U}}$, it has to be fixed by the action of $A$ and $\sigma$. Hence, we have that $A_{1,1}\neq 0$ and for every $\zeta\in T$ there exists $\eta\in S$ such that 
    \[ A_{1,1} \zeta^\sigma=\eta,\]
    which implies that $A_{1,1}T^\sigma=S$, a contradiction.
\end{proof}

\begin{remark}
    Note that the existence of $S$ and $T$ with the property that there exist no $\alpha \in \fqm$ and $\sigma \in \mathrm{Aut}(\fqm)$ such that $S=\alpha T^\sigma$ is guaranteed from the following fact. 
    This condition is preserved by considering their duals, i.e. $S=\alpha T^\sigma$ if and only if \begin{equation}\label{eq:SperpTperp}
    S^\perp=\alpha^{-1}(T^\perp)^\sigma.
    \end{equation}
    If $n$ is even and we choose $S^\perp=\F_{q^2}$ and $T^\perp=\la 1,\lambda\ra_{\fq}$ with $\lambda \in \F_{q^2}\setminus \fq$, since $\F_{q^2}$ and $\langle 1, \lambda \rangle_{\fq}$ cannot be obtained one from the other as in \eqref{eq:SperpTperp}.
\end{remark}

\begin{remark}
    As you have seen, we did not consider the case in which $i=\in \{m-1,m\}$.
    Indeed, in Section \ref{sec:classm-1club}, we will characterize the $(m-1)$-clubs having maximum rank. As a consequence, the two constructions presented in this section, when $i=m-1$, are $\Gamma\mathrm{L}(k,q^m)$-equivalent.
    Clearly, the same happens when $i=m$ as all the constructions of $m$-clubs of maximum rank are of the cone construction form.
\end{remark}

\subsection{Constructions for $k$ even}
In this section, we will give a construction of $i$-clubs of maximum rank in $\mathrm{PG}(k-1,q^m)$ when $k$ is even. 
When $k\geq 6$ we can use a construction very similar to that of Proposition \ref{con:conekodd} and following the same strategy when $k>4$ we can prove the following.

\begin{theorem} \label{th:inequivalencekeven}
Let $k \geq 6$ and $m$ be two positive integers such that $k$ and $m$ are even numbers.
    Consider
    \[ U_1=\{(x+\zeta,x^q,x^{q^2})\mid x\in \mathbb{F}_{q^m}, \zeta\in S\}, \]
    for an $\fq$-subspace $S$ of $\fqm$ with dimension $i$, and let $L_{U_2}$ be a maximum scattered $\fq$-linear set in the subspace of $\mathrm{PG}(k-1,q^m)$ having equations $X_0=X_1=X_2=0$.
    The subspace $U=U_1+U_2$ defines an $i$-club in $\mathrm{PG}(k-1,q^m)$ of rank $m(k-1)/2+i$.
    Moreover, if $i\leq m-3$ then $L_U$ is not $\mathrm{\Gamma L}(k,q^m)$-equivalent to an $i$-club of the form as in Proposition \ref{prop:coneconstruct}.
\end{theorem}
\begin{proof}
    The first part can be proved similarly as has been done in Proposition \ref{con:conekodd}.
    Again, following the same approach in Theorem \ref{thm:conekoddhyper}, we have that
    \[ w_{L_U}(H) \in \{m(k-3)/2+i,m(k-3)/2+i+1,m(k-3)/2+i+2\}, \]
    for every hyperplane $H$ in $\mathrm{PG}(k-1,q^m)$.
\end{proof}

As we have already observed, for the case $k=4$ the above theorem does not hold. The main problem is that the subspace that we add in the direct sum with $U_1$ is too small to get the maximum value of the rank. When $k=4$ we are only able to provide a new construction of $i$-club when $i=m/2$.

\begin{theorem}\label{thm:constrm/2-club}
    Let $n$ be an even natural number with $m\geq 4$.
    The subspace
    \[ U=\{(x,\mathrm{Tr}_{q^m/q^{m/2}}(x),y,y^q) \colon x,y \in \fqm \}\]
    defines an $m/2$-club in $\mathrm{PG}(3,q^m)$ of rank $2m$ which is not $\mathrm{\Gamma L}(4,q^m)$-equivalent to those in Proposition \ref{prop:coneconstruct}.
\end{theorem}
\begin{proof}
    Since $\{(x,\mathrm{Tr}_{q^m/q^{m/2}}(x),0,0) \colon x \in \fqm \}$ is an $m/2$-club and $\{(0,0,y,y^q) \colon x,y \in \fqm \}$ defines a maximum scattered $\fq$-linear set in the line having equation $X_0=X_1=0$, by Lemma \ref{le:iclubandscatered} we have that $U$ defines an $m/2$-club. Now, suppose that $L_U$ is $\mathrm{\Gamma L}(4,q^m)$-equivalent to an $m/2$-club in Proposition \ref{prop:coneconstruct}, i.e. $U$ is $\mathrm{\Gamma L}(4,q^m)$-equivalent to $\overline{U}$ where 
    $$\overline{U}=U'+U'',$$
    $L_{U'}$ is a maximum scattered $\fq$-linear set contained in a plane $\pi$ of $\mathrm{PG}(3,q^m)$ and $U''$ is an $\fq$-subspace of dimension $m/2$ contained in $\langle x\rangle_{\mathbb{F}_{q^m}}$, for some point $\langle x\rangle_{\mathbb{F}_{q^m}}\notin \pi$.
    By definition, there exists at least one line having weight $m$ in $L_U$, and so there exists at least one line $\ell=\mathrm{PG}(W,\fqm)$ of weight $m$ with respect to $L_{\overline{U}}$.
    Observe that if $\overline{W}=U\cap W$ then we have
    \[ \dim_{\fq}(\overline{W}\cap U')=m+\frac{3m}2-\dim_{\fq}(\overline{W}+ U')\geq m/2,\]
    implying that $w_{L_{U'}}(\ell) \geq m/2$ and so $\ell$ must be contained in $\pi$.
    Therefore, $w_{L_{U'}}(\ell)=m$, which contradicts the fact that $L_{U'}$ is a maximum scattered linear set contained in $\pi$ and the weight of the line with respect to $L_{U'}$ can only be $m/2$ or $m/2+1$ according to Theorem \ref{th:inter}.
\end{proof}
    
\section{Classification of $(m-1)$-clubs in $\PG(k-1,q^m)$ of rank $\frac{m(k+1)}{2}-1$}\label{sec:classm-1club}

In this section, we classify \((m-1)\)-clubs having the maximum possible rank \(\frac{m(k+1)}{2} - 1\). This extends the classification of $(m-1)$-clubs in $\mathrm{PG}(1,q^m)$ (see e.g. \cite[Theorem 3.7]{csajbok2018classes} or \cite[Theorem 2.3]{de2016linear}) and $(m-1)$-clubs in $\mathrm{PG}(2,q^m)$ of rank $m+1$ in \cite[Theorem 5]{lunardon2000blocking}.

We begin with the following preliminary lemma.

\begin{lemma}\label{obs:weighdistperp}
Let $L_U$ be an $(m-1)$-club of rank $\frac{m(k+1)}{2} - 1$ in $\mathrm{PG}(k-1,q^m)$.
The dual \(L_{U^{\perp'}}\) of \(L_U\) is an \(\mathbb{F}_q\)-linear set of rank \(\frac{m(k-1)}{2} + 1\) such that 
\[
w_{L_{U^{\perp'}}}(H) \in \left\{1 + \frac{m(k-3)}{2},\; 2 + \frac{m(k-3)}{2},\; n + \frac{m(k-3)}{2} \right\},
\]
for any hyperplane \(H\) of \(\mathrm{PG}(k-1, q^m)\), and there exists exactly one hyperplane of weight \(m + \frac{m(k-3)}{2}\). In particular $\langle U^{\perp'} \rangle_{\F_{q^m}}=\mathbb{V}$.
\end{lemma}
\begin{proof}
   Since \(L_{U^{\perp'}}\) is the dual of a linear set of rank \(\frac{m(k+1)}{2} - 1\), it follows that \(L_{U^{\perp'}}\) has rank
\[
km - \left( \frac{m(k+1)}{2} - 1 \right) = \frac{m(k - 1)}{2} + 1.
\]
We now compute the weights of hyperplanes with respect to \(L_{U^{\perp'}}\). To this end, we use Proposition~\ref{prop:dualityproperties} (v), and consider a hyperplane \(H = \mathrm{PG}(W, \mathbb{F}_{q^m})\) of \(\mathrm{PG}(k-1, q^m)\). We then have:
\[
\begin{split}
\dim_{\mathbb{F}_q}(U^{\perp'} \cap W) &= \dim_{\mathbb{F}_q}(U \cap W^{\perp'}) + km - \left( \frac{m(k+1)}{2} - 1 \right) - m \\
&= \dim_{\mathbb{F}_q}(U \cap W^{\perp'}) + \frac{m(k - 3)}{2} + 1.
\end{split}
\]

Since \(L_U\) is an \((m - 1)\)-club, we know that \(\dim_{\mathbb{F}_q}(U \cap W^{\perp'}) \in \{0, 1, m - 1\}\). Therefore, the weight of \(H\) with respect to \(L_{U^{\perp'}}\) satisfies:
\[
w_{L_{U^{\perp'}}}(H) \in \left\{1 + \frac{m(k - 3)}{2},\; 2 + \frac{m(k - 3)}{2},\; m + \frac{m(k - 3)}{2} \right\}.
\]

Finally, since \(L_U\) is an \((m - 1)\)-club, there exists exactly one point of weight \(m - 1\), which corresponds (under duality) to a unique hyperplane of weight \(m + \frac{m(k - 3)}{2}\). This completes the proof.
\end{proof}

Using the result above, we can characterize the structure of $L_{U^{\perp'}}$. This will allows us to give the final characterization of $(m-1)$-clubs.

\begin{lemma}\label{le:analysedualset}
Let $L_U$ be an $(m-1)$-club of rank $\frac{m(k+1)}{2} - 1$ in $\mathrm{PG}(k-1,q^m)$.
Let \(\overline{H} = \mathrm{PG}(\overline{W}, \mathbb{F}_{q^m})\) be the unique hyperplane of weight \(m + \frac{m(k - 3)}{2}\) with respect to \(L_{U^{\perp'}}\). For any point \(Q = \langle w \rangle_{\mathbb{F}_{q^m}} \in L_{U^{\perp'}}\), with \(w \in U^{\perp'}\) and \(Q \notin \overline{H}\), we have:
\[
U^{\perp'} = (U^{\perp'} \cap \overline{W}) \oplus \langle w \rangle_{\mathbb{F}_q} \,\,\,\text{and}\,\,\, \la U^{\perp'} \cap \overline{W} \ra_{\fqm}\cap \langle w \rangle_{\fqm}=\{0\}.
\]
\end{lemma}
\begin{proof}
Let \(V = U^{\perp'} \cap \overline{H}\). Since \(\overline{H} = \mathrm{PG}(\overline{W}, \mathbb{F}_{q^m})\) is an hyperplane of weight \(m + \frac{m(k - 3)}{2}\) with respect to \(L_{U^{\perp'}}\), we have
\[
\dim_{\mathbb{F}_q}(V) = w_{L_{U^{\perp'}}}(\overline{H}) = \frac{m(k - 1)}{2}.
\]
Now, let \(Q = \langle w \rangle_{\mathbb{F}_{q^m}} \in L_{U^{\perp'}}\) be a point with \(w \in U^{\perp'}\) and \(Q \notin \overline{H}\). Since \(Q \notin \overline{H}\), it follows that $w \notin V$. Given that \(\dim_{\mathbb{F}_q}(U^{\perp'}) = \frac{m(k - 1)}{2} + 1\), we immediately obtain the decomposition:
\[
U^{\perp'} = V \oplus \langle w \rangle_{\mathbb{F}_q}.
\]
\end{proof}

This result leads to the following classification theorem, proving that any maximum $(m-1)$-club in $\PG(k-1,q^m)$ is as in the construction described in Proposition \ref{prop:coneconstruct}.

\begin{theorem} \label{th:classn-1clubmaximum}
Let \(L_U\) be an \((m - 1)\)-club in \(\mathrm{PG}(k - 1, q^m)\) of rank \(\frac{m(k + 1)}{2} - 1\). Then, up to $\Gamma \mathrm{L}(k,q^m)$-equivalence, we have
\[
U = U' \oplus U'',
\]
where \(L_{U'}\) is a maximum scattered \(\mathbb{F}_q\)-linear set contained in a hyperplane \(H\) of \(\mathrm{PG}(k - 1, q^m)\), and \(U''\) is an \((m - 1)\)-dimensional \(\mathbb{F}_q\)-subspace of \(\langle x \rangle_{\mathbb{F}_{q^m}}\), for some \(\langle x \rangle_{\mathbb{F}_{q^m}} \notin H\). Also, $L_{U^{\perp'}}$ is a scattered $\F_q$-linear set. 
\end{theorem}
\begin{proof}
By Lemma~\ref{le:analysedualset}, we know that
\[
U^{\perp'} = (U^{\perp'} \cap \overline{W}) \oplus \langle w \rangle_{\mathbb{F}_q},
\]
where \(\overline{W}\) is such that \(\overline{H} = \mathrm{PG}(\overline{W}, \mathbb{F}_{q^m})\) is the unique hyperplane of weight \(m + \frac{m(k - 3)}{2}\) with respect to \(L_{U^{\perp'}}\), and \(w \in U^{\perp'}\) is such that the corresponding point \(Q = \langle w \rangle_{\mathbb{F}_{q^m}}\) does not lie in \(\overline{H}\).
We now fix a coordinate system \((x_0, \ldots, x_{k-1})\) on \(\mathrm{PG}(k - 1, \mathbb{F}_{q^m})\) such that \(\overline{H}\) is the hyperplane defined by the equation \(x_{k-1} = 0\), and the point \(Q = \langle w \rangle_{\mathbb{F}_{q^m}}\) corresponds to \(\langle (0, \ldots, 0, 1) \rangle_{\mathbb{F}_{q^m}}\).
We consider the standard inner product as a sesquilinear form \(\sigma \colon \mathbb{F}_{q^m}^k \times \mathbb{F}_{q^m}^k \rightarrow \mathbb{F}_{q^m}\), and define
\[
\sigma' \colon \mathbb{F}_{q^m}^k \times \mathbb{F}_{q^m}^k \rightarrow \mathbb{F}_q, \quad \sigma'(u,v) = \mathrm{Tr}_{q^n/q}(\sigma(u,v)).
\]
We let \(\perp'\) denote the orthogonality map defined by \(\sigma'\). Note that a different choice of sesquilinear form \(\sigma\) yields \(\Gamma \mathrm{L}(k, \mathbb{F}_{q^m})\)-equivalent subspaces \(U^{\perp'}\) (cf.\ Proposition~\ref{prop:dualityproperties}).

Now, let \(\perp^\star\) and \(\perp^{\star\star}\) denote the duality map induced by the bilinear form $\sigma'$ restricted to the hyperplane \(\overline{H} = \mathrm{PG}(\overline{W}, \mathbb{F}_{q^m})\) defined by \(x_{k-1} = 0\), and the subspace defined by \(x_0 = \cdots = x_{k-2} = 0\), respectively.  Hence, it is easy to observe that:
\[
U = (U^{\perp'})^{\perp'} = U'^{\perp^\star} \oplus U''^{\perp^{\star\star}},
\]
where \(U' = U^{\perp'} \cap \overline{W}\) and \(U'' = \langle w \rangle_{\mathbb{F}_q}\).
Clearly, 
\[ w_{L_U}(\la (0,\ldots,0,1)\ra_{\fqn})=\dim_{\fq}(U''^{\perp**})=m-1, \]
and $\dim_{\fq}(U'^{\perp^\star})=m(k-1)/2$.
Note that $L_{U'^{\perp^\star}}$ is scattered with 
\[
\begin{array}{rl}
\mathrm{Rank}(L_{U'^{\perp^\star}}) & =\dim_{\F_q}\left(U'^{\perp^\star}\right) \\
& = (k-1)m-\dim_{\F_q}\left(U'\right) \\
& = \frac{(k-1)m}{2}.
\end{array}
\]
Indeed, $L_U$ is an $(m-1)$-club and the only point of weight $m-1$ is $\la(0,\ldots,0,1)\ra_{\fqm}$ which is not in $L_{U'^{\perp^\star}}$.
This proves the first part of the assertion. It remains to show that \(L_{U^{\perp'}}\) is a scattered $\F_q$-linear set. We do this by studying the weight of its points.
First, consider any point \(P = \langle v \rangle_{\mathbb{F}_{q^m}} \in L_{U^{\perp'}} \setminus L_{U'}\). We claim that such a point also has weight $1$. Indeed, if this point has weight at least $2$, since $\langle v \rangle_{\fqm} \cap \langle U' \rangle_{\fqm}=\{0\}$, then we would have:
\[
\dim_{\mathbb{F}_q}(\langle v, U' \rangle_{\mathbb{F}_{q^m}} \cap U^{\perp'}) \geq \dim_{\mathbb{F}_q}(\langle v \rangle_{\mathbb{F}_{q^m}} \cap U^{\perp'}) + \dim_{\mathbb{F}_q}(U' \cap U^{\perp'}) > \frac{m(k - 1)}{2} + 1,
\]
which contradicts the fact that \(\dim_{\mathbb{F}_q}(U^{\perp'}) = \frac{m(k - 1)}{2} + 1\). It remains to show that \(L_{U'}\) is scattered. This follows by the fact that $L_{U'}$ is contained in the hyperplane $\bar{H}$ and $L_{U'^{\perp*}}$ is a maximum scattered linear set in $\bar{H}$, cf. \cite[Theorem 3.5]{polverino2010linear}. This concludes the proof. 
\end{proof}

We finally observe that a similar classification result cannot be provided for maximum $i$-club, when $i\leq m-2$, as we have shown that there are inequivalent constructions when $i\leq m-2$.

\section{Three weight rank-metric codes}\label{sec:threeweights}

We will show how the geometry of the three-weight rank-metric codes is much more rich than the one of two-weight rank-metric codes.
Indeed, a two-weight rank-metric code has been geometrically characterized as either scattered or the dual of a scattered linear set; see \cite{zullo2024two,pratihar2024antipodal}.

We will now give some examples of three-weight codes arising from $\fq$-subspaces in $\F_{q^m}^k$ defining some known linear sets.
We split the analysis according to: dimension of the code equals to two, dimension of the code equal to three, dimension of the code greater than three and three-weight codes arising from clubs.
In all of the next subsections we will assume that $\langle U \rangle_{\F_{q^m}}=\F_{q^m}^k$ in order to have the correspondence with the rank-metric codes.

\subsection{$k=2$}

We introduce two constructions of $\fq$-subspaces of $\F_{q^m}^2$ yielding $[n,2]_{q^m/q}$ codes with exactly three nonzero weights, characterized by the presence of one or two distinct $1$-dimensional $\F_{q^m}$-subspaces whose intersections with the given subspace have $\fq$-dimension strictly greater than one, while all other such intersections are of dimension at most one.

\begin{construction}[\textbf{Club}]\label{exiclub}
Let $U$ be an $\fq$-subspace of $\F_{q^m}^2$ with dimension $n$ such that there exists exactly one $1$-dimensional $\F_{q^m}$-subspace $\langle v \rangle_{\F_{q^m}}$ of $\F_{q^m}^2$ such that $\dim_{\fq}(U\cap \langle v \rangle_{\F_{q^m}})=i>1$ and for all the remaining $1$-dimensional $\F_{q^m}$-subspaces the intersection with $U$ has dimension either $0$ or $1$.
An $[n,2]_{q^m/q}$ code associated with $U$ has dimension two and nonzero weights $n-i, n-1,n$.
\end{construction}

For the known results and constructions, see Section \ref{sec:genclubs}.

\begin{construction}[\textbf{Complementary weights}]\label{excompl}
Let $U$ be an $\fq$-subspace of $\F_{q^m}^2$ with dimension $n$ such that there exists exactly two distinct $1$-dimensional $\F_{q^m}$-subspace $\langle v_1 \rangle_{\F_{q^m}}$ and $\langle v_2 \rangle_{\F_{q^m}}$ of $\F_{q^m}^2$ such that $\dim_{\fq}(U\cap \langle v_j \rangle_{\F_{q^m}})=s>1$, for $j \in \{1,2\}$, $2s=n$ and for all the remaining $1$-dimensional $\F_{q^m}$-subspaces the intersection with $U$ has dimension either $0$ or $1$.
An $[n,2]_{q^m/q}$ code associated with $U$ has dimension two and nonzero weights $n-s, n-1,n$.
\end{construction}

Constructions of such subspaces can be found in \cite{napolitano2022linear}. 
In \cite[Theorem 4.4]{napolitano2022linear}, it has been proved that, under the assumptions of Construction \ref{excompl} and $n=m$, the only possible value for $s$ is $m/2$. 
An example is described in \cite[Corollary 4.9]{napolitano2022linear}:
let $m=2t$, $\mu \in \F_{q^t}^*$ such that $\N_{q^t/q}(\mu) \neq 1$ and $\N_{q^{t}/q}(-\xi^{q^t+1}\mu)\neq (-1)^t$, 
with $\xi \in \F_{q^{2t}}\setminus\F_{q^t}$.
Then
\[ U=\{ (u+\xi \mu u^{q},v+\xi v^{q}) \colon u,v \in \F_{q^t} \}, \]
is an example of Construction \ref{excompl}.

In \cite{de2022weight}, some computational results show the existence for $q \in \{2,3\}$ of $\fq$-subspaces $U$ in $\F_{q^5}^2$ of dimension $5$ such that there exists exactly two $1$-dimensional $\F_{q^m}$-subspace $\langle v_1 \rangle_{\F_{q^m}}$ and $\langle v_2 \rangle_{\F_{q^m}}$ of $\F_{q^m}^2$ such that $\dim_{\fq}(U\cap \langle v_j \rangle_{\F_{q^m}})=2$, for $j \in \{1,2\}$, and for all the remaining $1$-dimensional $\F_{q^m}$-subspaces the intersection with $U$ has dimension either $0$ or $1$.
An $[5,2]_{q^5/q}$ code associated with $U$ has nonzero weights $3,4,5$.
These examples cannot be extended to $q>3$ due to \cite[Main Theorem]{de2022weight}.

\subsection{$k=3$}

We present constructions of $\mathbb{F}_q$-subspaces of $\mathbb{F}_{q^m}^3$ whose intersection profiles with $2$-dimensional $\mathbb{F}_{q^m}$-subspaces are highly constrained, resulting in codes with exactly three nonzero weights. These constructions include $\fq$-subspaces that are scattered with respect to the lines, as well as subspaces admitting a small number of higher-dimensional intersections with $\F_{q^m}$-lines, and are closely related to blocking sets of Rédei type.

\begin{construction}[\textbf{Scattered spaces with respect to the lines}]
Let $U$ be an $\fq$-subspace of $\F_{q^m}^3$ with dimension $n$ such that for every $2$-dimensional $\F_{q^m}$-subspace $W$ then $\dim_{\fq}(W\cap U)\leq 2$. An $[n,3]_{q^m/q}$ code associated with $U$ has dimension three and nonzero weights $n-2,n-1,n$.
\end{construction}

Constructions of such subspaces can be found in \cite{csajbok2021generalising,lunardon2017mrd,sheekeyVdV}.
Let $\gcd(s,m)=1$ and $\delta \in \F_{q^m}$ such that $\N_{q^m/q}(\delta)\ne (-1)^m$.
An example is given by
\[ U=\{(x-\delta x^{q^{3s}},x^{q^s},x^{q^{2s}})\colon x \in \mathbb{F}_{q^m}\}\subseteq \mathbb{F}_{q^m}^3, \]
which corresponds to Generalized Twisted Gabidulin codes; see \cite{sheekey2016new,lunardon2018generalized}.

\begin{construction}[\textbf{Scattered blocking sets of R\'edei type}]
Let $U$ be an $\fq$-subspace of $\F_{q^m}^3$ with dimension $m+1$ such that there exists a $2$-dimensional $\F_{q^m}$-subspace $\overline W$ with the property that $\dim_{\fq}(\overline W\cap U)=m$ and $U$ defines a scattered space. This implies that $\dim_{\fq}(U\cap W)\in \{1,2,m\}$ for any $2$-dimensional $\F_{q^m}$-subspace $W$ of $\F_{q^m}^3$. An $[m+1,3]_{q^m/q}$ code associated with $U$ has dimension three and nonzero weights $1,m-1,m$.
\end{construction}

Constructions of such subspaces can be found in e.g. \cite{lunardon2001blocking,Polverinothesis}.
An example is given by the following subspace
\[ U=\{(x-\delta x^{q^{2s}},x^{q^s},\alpha)\colon x \in \F_{q^m},\alpha \in \fq\}\subseteq \mathbb{F}_{q^m}^3, \]
where $\gcd(s,m)=1$ and $\N_{q^m/q}(\delta)\ne 1$.

\begin{construction}[\textbf{Complementary weights}]\label{con:complplane}
Let $U$ be an $\fq$-subspace of $\F_{q^m}^3$ with dimension $n$ such that there exists exactly three distinct $1$-dimensional $\F_{q^m}$-subspace $\langle v_1 \rangle_{\F_{q^m}}$, $\langle v_2 \rangle_{\F_{q^m}}$ and $\langle v_3 \rangle_{\F_{q^m}}$ of $\F_{q^m}^3$ such that $\dim_{\fq}(U\cap \langle v_j \rangle_{\F_{q^m}})=s>1$, for $j \in \{1,2,3\}$, and for all the remaining $1$-dimensional $\F_{q^m}$-subspaces the intersection with $U$ has dimension either $0$ or $1$.
An $[3m-n,3]_{q^m/q}$ code associated with $U^{\perp'}$ has dimension three and nonzero weights $m-s,m-1,m$.
\end{construction}

Constructions of such subspaces can be found in \cite{zullo2021multi}.
An example is given in \cite[Corollary 6.8]{zullo2021multi}:
let $q\geq 4$, $m=2t$, $\mu_1,\mu_2,\mu_3 \in \F_{q^t}^*$ such that $\N_{q^t/q}(\mu_i) \neq \N_{q^t/q}(\mu_j)$ and $\N_{q^{t}/q}(-\xi^{q^t+1}\mu_i\mu_j)\neq (-1)^t$ for any $i,j \in \{1,2,3\}$ with $i\ne j$, 
with $\xi \in \F_{q^{2t}}\setminus\F_{q^t}$.
Then
\[ U=\{ (u+\xi \mu_1 u^{q},v+\xi \mu_2 v^{q},z+\xi \mu_3 z^{q}) \colon u,v,z \in \F_{q^t} \} \subseteq \F_{q^m}^3. \]
In this case $s=t=m/2$.

\subsection{$k>3$}

We consider families of $\mathbb{F}_q$-subspaces of $\mathbb{F}_{q^m}^k$ whose structural constraints on intersections with $\mathbb{F}_{q^m}$-hyperplanes induce codes with exactly three nonzero weights. In particular, we focus on \emph{maximum scattered} subspaces with respect to $\F_{q^m}$-lines, i.e., $\fq$-subspaces of dimension $km/3$ intersecting every $2$-dimensional $\F_{q^m}$-subspace in dimension at most $2$, and on duals of subspaces exhibiting a controlled number of high-dimensional intersections with $1$-dimensional $\F_{q^m}$-subspaces. These constructions generalize previous results for $k=3$.

\begin{construction}[\textbf{Maximum scattered spaces with respect to the lines}]\label{con:maxscatlines}
Let $U$ be an $\fq$-subspace of $\F_{q^m}^k$ with dimension $km/3$ such that for every $2$-dimensional $\F_{q^m}$-subspace $W$ then $\dim_{\fq}(W\cap U)\leq 2$. An $[n,k]_{q^m/q}$ code $\C$ associated with $U$ has dimension $k$ and nonzero weights $m-2,m-1,m$.
\end{construction}

Indeed, in \cite[Theorem 2.7]{csajbok2021generalising} it has been proved that if $U$ is an $\fq$-subspace as in Construction \ref{con:maxscatlines}, then
\[ \dim_{\fq}(U\cap H) \in \left\{ \frac{km}3-m,\frac{km}3-m+1, \frac{km}3-m+2 \right\}, \]
for any $(k-1)$-dimensional $\F_{q^m}$-subspace of $\F_{q^m}^k$.
Moreover, by \cite[Theorem 7.1]{zini2021scattered} for each of these values there exists at least one hyperplane meeting $U$ in this dimension
(actually they can be characterized via these intersection numbers, see \cite[Corollary 5.3]{zini2021scattered}). Let \( G \) be a generator matrix of \( \mathcal{C} \) whose columns form an \( \mathbb{F}_q \)-basis of \( U \). Then, by Theorem~\ref{th:connection}, the rank weight of a codeword \( xG \in \mathcal{C} \), with \( x \in \mathbb{F}_{q^m}^k \), is given by
\[
w(xG) = \frac{km}{3} - \dim_{\mathbb{F}_q}(U \cap x^{\perp}) \in \{m-2, m-1, m\}.
\]

Constructions of such subspaces can be found in \cite{csajbok2021generalising}.
An example is given by the following subspace:
let $k=3t$, then
\[ U=\{ (x_1,x_1^q,x_1^{q^2},x_2,x_2^q,x_2^{q^2},\ldots,x_t,x_t^q,x_t^{q^2}) \colon x_1,\ldots,x_t \in \F_{q^m} \}\subseteq \F_{q^m}^{k}, \]
see also Section 6 of \cite{napolitano2020codes}.

\begin{construction}[\textbf{Complementary weights}] \label{const:complweightsk>3}
Let $U$ be an $\fq$-subspace of $\F_{q^m}^k$ with dimension $n$ such that there exists exactly $k$ distinct $1$-dimensional $\F_{q^m}$-subspace $\langle v_1 \rangle_{\F_{q^m}},\ldots,\langle v_k \rangle_{\F_{q^m}}$ of $\F_{q^m}^k$ such that $\dim_{\fq}(U\cap \langle v_j \rangle_{\F_{q^m}})=s>1$, for $j \in \{1,\ldots,k\}$, and for all the remaining $1$-dimensional $\F_{q^m}$-subspaces the intersection with $U$ has dimension either $0$ or $1$.
A $[km-n,k]_{q^m/q}$ code associated with $U^{\perp'}$ has dimension $k$ and nonzero weights $m-s,m-1,m$.
\end{construction}

Constructions of such subspaces can be found in \cite{zullo2021multi}, which extends Construction \ref{con:complplane} to higher dimensions.
An example is given in \cite[Corollary 6.8]{zullo2021multi}:
let $q\geq k+1$, $m=2t$, $\mu_1,\ldots,\mu_k \in \F_{q^t}^*$ such that $\N_{q^t/q}(\mu_i) \neq \N_{q^t/q}(\mu_j)$ and $\N_{q^{t}/q}(-\xi^{q^t+1}\mu_i\mu_j)\neq (-1)^t$ for any $i,j \in \{1,\ldots,k\}$ with $i\ne j$, 
with $\xi \in \F_{q^{2t}}\setminus\F_{q^t}$.
Then
\[ U=\{ (u_1+\xi \mu_1 u_1^{q},\ldots,u_k+\xi \mu_k u_k^{q}) \colon u_1,\ldots,u_k \in \F_{q^t} \}, \]
is an example of Construction \ref{const:complweightsk>3}.
In this case $s=m/2$.

\subsection{Three-weight rank-metric codes from clubs}

In order to prove an upper bound on the rank of an $i$-club, we have used MacWilliams identities on the code associated with dual of the considered club; see Proposition \ref{prop:codeparam}. Clearly, also the converse holds: if $\C$ is an $[n,k,d]_{q^m/q}$ with $A_{m}\ne 0$, $A_{m-1}\ne 0$ and $A_d=q^m-1$, then $\C$ is associated with the dual of an $i$-club.
Therefore, together with Proposition \ref{prop:codeparam}, we have the following one-to-one correspondence.

\begin{theorem}
    Let $\C$ be an $[n,k,d]_{q^m/q}$ code and $U$ an its associated system. 
    We have that $\C$ is a three-weight code with $A_{m}\ne 0$, $A_{m-1}\ne 0$ and $A_d=q^m-1$ if and only if $L_{U^\perp}$ is an $(m-d)$-club of rank $km-n$ in $\mathrm{PG}(k-1,q^m)$.
    In particular,
    \[ n\leq  
    \begin{cases}
        \frac{mk}2 & \text{if either }\,\, d\geq m/2 \,\, \text{or } k=2,\\
        \frac{m(k+1)}2-m+d & \text{if }\,\, 1\leq d\leq m/2.
    \end{cases}
    \]
\end{theorem}
\begin{proof}
    This is a direct consequence of the fact that for any hyperplane $H$, by \eqref{eq:relweight}, we have that
    \[ \dim_{\fq}(U\cap H)\in\{n-m,n-m+1,n-d\}, \]
    and there exists only one hyperplane such that $\dim_{\fq}(U\cap H)=n-d$.
    By using Proposition \ref{prop:dualityproperties} we get the assertion.
    The converse is Proposition \ref{prop:codeparam}.
    The last part follows from Corollary \ref{cor:boundiclub}.
\end{proof}

We can therefore characterize the three-weight codes associated with the dual of $(m-1)$-clubs making use of the results in Section \ref{sec:classm-1club}.
We start by determining the weight distribution of the following class of three-weight rank-metric codes.

\begin{theorem}
    Let $\C$ be a three-weight code with parameters $[m(k-1)/2+1,k,1]_{q^m/q}$ and $A_m\ne 0$, $A_{m-1}\ne 0$ and $A_1=q^m-1$.
    We have that $\C$ is equivalent to
    \[ \C= \C_1 \oplus \C_2, \]
    where $\C_1=\overline{\C}_1\times\{0\}$, $\overline{\C}_1$ is an MRD code with parameters $[m(k-1)/2,k-1,m-1]_{q^m/q}$ and $\C_2=\la (0,\ldots,0,1)\ra_{\fqm}$.
\end{theorem}
\begin{proof}
    Let $U$ be a system associated with $\C$.
    Since, by the above theorem, $U^\perp$ is an $(m-1)$-club of maximum rank, we can apply Lemma \ref{le:analysedualset} on $L_{U^\perp}$. Hence, there exists an hyperplane \(\overline{H} = \mathrm{PG}(\overline{W}, \mathbb{F}_{q^m})\) of weight \(m + \frac{m(k - 3)}{2}\) with respect to \(L_{U}\) and a point \(Q = \langle w \rangle_{\mathbb{F}_{q^m}} \in L_{U}\), with \(w \in U\) and \(Q \notin \overline{H}\) for which
    \[
    U = (U \cap \overline{W}) \oplus \langle w \rangle_{\mathbb{F}_q}.
    \]
    We coordinatize $\mathrm{PG}(k-1,q^m)$ in such a way that $\overline{H}$ has equation $x_{k-1}=0$ and $Q$ has coordinates $(0,\ldots,0,1)$.
    Note that the dual $(U \cap \overline{W})^{\perp^*}$ of $U \cap \overline{W}$ restricted to $\overline{W}$, as shown in the proof of Theorem \ref{th:classn-1clubmaximum}, is a scattered $\fq$-subspace contained in $\overline{W}$ of dimension $m(k-1)/2$. By \cite[Theorem 3.5]{polverino2010linear}, we have that $U \cap \overline{W}$ is a scattered $\fq$-subspace of $\overline{W}$.
    Therefore, if we consider the code $\overline{C}_1\subseteq \fqm^{m(k-1)/2}$ of dimension $k-1$ associated with $U \cap \overline{W}$ (by forgetting about the last coordinate), by \cite[Section 3]{zini2021scattered}, $\overline{C}_1$ is an MRD code.
    The assertion follows by writing down the generator matrix associated with the subspace $U$ and by Theorem \ref{th:connection}.
\end{proof}

We can also consider the code associated an $(m-1)$-club (instead of the one associated with the dual) and we get a two-weight rank-metric code. Indeed, for $(m-1)$-club linear set $L_U$ of rank $\frac{m(k+1)}{2} - 1$ in $\mathrm{PG}(k-1,q^m)$, using the fact that $L_{U^{\perp'}}$ is scattered, we can calculate the weights of the hyperplanes in $L_U$. 
\begin{lemma}\label{lem:weightslinesclass}
 Let $L_U$ be an $(m-1)$-club of rank $\frac{m(k+1)}{2} - 1$ in $\mathrm{PG}(k-1,q^m)$. For every hyperplane \(H\) of \(\mathrm{PG}(k-1, q^m)\), the weight of \(H\) with respect to \(L_U\) satisfies
\[
w_{L_U}(H) \in \left\{ \frac{m(k-1)}{2},\; \frac{m(k-1)}{2} - 1 \right\}.
\]
Moreover, there are exactly \(\frac{q^{\frac{m(k-1)}{2} + 1} - 1}{q - 1}\) hyperplanes of weight \(\frac{m(k-1)}{2}\) and all the rest have weight \(\frac{m(k-1)}{2}-1\).
\end{lemma}
\begin{proof}
We apply Proposition~\ref{prop:dualityproperties} (v), which states that for every hyperplane \(H\) of \(\mathrm{PG}(k-1, q^m)\), the weight of \(H\) with respect to \(L_U\) satisfies:
\[
\begin{split}
w_{L_U}(H) &= w_{L_{U^{\perp'}}}(H^{\perp'}) - m + \left( \frac{m(k+1)}{2} - 1 \right) \\
           &= w_{L_{U^{\perp'}}}(H^{\perp'}) + \frac{m(k - 1)}{2} - 1.
\end{split}
\]

From \Cref{th:classn-1clubmaximum}, we know that \(L_{U^{\perp'}}\) is a scattered linear set, which implies that for every point \(P=H^{\perp'}\), the weight \(w_{L_{U^{\perp'}}}(P) \in \{0, 1\}\). Therefore, the weights of the hyperplanes with respect to \(L_U\) are:
\[
w_{L_U}(H) \in \left\{ \frac{m(k - 1)}{2} - 1,\; \frac{m(k - 1)}{2} \right\}.
\]

Finally, since each hyperplane \(H\) such that \(w_{L_{U^{\perp'}}}(H^{\perp'}) = 1\) corresponds to a point of weight one in \(L_{U^{\perp'}}\), and there are \(\frac{q^{\frac{m(k-1)}{2} + 1} - 1}{q - 1}\) such points, this gives the number of hyperplanes of weight \(\frac{m(k - 1)}{2}\), completing the proof.
\end{proof}

Therefore, combining the above result and \eqref{eq:relweight}, we obtain the following.

\begin{corollary}
    Let $L_U$ be an $(m-1)$-club of rank $\frac{m(k+1)}{2} - 1$ in $\mathrm{PG}(k-1,q^m)$. If $\C$ be a code associated with $U$, then $\C$ is a two-weight $[m(k+1)/2-1,k,m-1]_{q^m/q}$ code with weight distribution
    \begin{itemize}
        \item $A_0=1$;
        \item $A_{m}=(q^m-1)\frac{q^{\frac{m(k-1)}{2} + 1} - 1}{q - 1}$;
         \item $A_{m-1}=q^{km}-1-A_{m-1}$;
        \item $A_j=0$ for any $j \notin\{0,m-1,m\}$.
    \end{itemize}
\end{corollary}

\section{Conclusion and open problems}

In this paper, we have investigated the theory of clubs in projective spaces. This family of linear sets has been extensively studied in the literature due to its connections with KM-arcs~\cite{de2016linear}, the direction problem \cite{napolitano2022clubs}, rank-metric codes \cite{sheekeyVdV,napolitano2022clubs}, representability of the free product of rank one uniform $q$-matroids \cite{alfarano2024free} and related topics.  
We provided upper bounds on the rank of an \( i \)-club in \(\PG(k-1, q^m)\), obtained by associating a rank-metric code to the linear set and applying the MacWilliams identities.  
We showed that these bounds are tight when \(k > 2\) and \(i \geq m/2\), and we also demonstrated that for \(i \leq m - 2\), there exist non-equivalent constructions attaining these bounds. In the special case \(i = m - 1\), we established a classification result.  
In the final section, we studied three-weight rank-metric codes. Unlike the case of two-weight codes, we showed that no unified geometric characterization is possible, due to the existence of several examples with distinct geometric structures. We then focused on three-weight codes arising as duals of clubs and provided a partial classification.

We highlight two main problems that we consider particularly interesting and for which new methodologies appear to be necessary:

\begin{itemize}
    \item 
  When \(k\) is even and \(k \geq 6\), in Theorem~\ref{th:inequivalencekeven}, we were able to construct examples of \(i\)-club linear sets of maximum rank, for every \(m/2 \leq i \leq m - 3\), that are not equivalent to the cone construction (Proposition~\ref{prop:coneconstruct}). However, for \(k = 4\), we provided non-equivalent examples only for \(i = m/2\), see Theorem~\ref{thm:constrm/2-club}. It would be interesting to find additional examples, particularly in the case \(i > m/2\).
    \item The most open and least understood case is when \(i < m/2\). In this case, we have an upper bound on the rank, but it is unclear how close this bound is to being optimal. This ambiguity has already appeared in the case \(k = 2\). A significant result in this direction was provided by De Boeck and Van de Voorde in~\cite{de2022weight}, where they proved that \(2\)-clubs of rank \(5\) in \(\PG(1, q^5)\) do not exist. This seems to suggest that the bound in this case is not optimal.
\end{itemize}

\section*{Acknowledgments}
The first author is very grateful for the hospitality of the Dipartimento Matematica e Fisica of University of Campania, he was visiting it during the development of this research in September 2024.
This research was partially supported by the Italian National Group for Algebraic and Geometric Structures and their Applications (GNSAGA - INdAM).

\bibliographystyle{abbrv}
\bibliography{biblio}
\end{document}